\documentclass[times]{article}

\usepackage{moreverb}

\usepackage{graphicx}
\usepackage{amsmath,amsfonts,amssymb,amsbsy,amsthm, subfigure,mathabx}
\usepackage{tikz,pgfplots}
\usepackage{bm}
\usepackage{comment}
\usepackage[a4paper,body={18cm,26cm},top=2cm]{geometry}

\usepackage[dvips,colorlinks,bookmarksopen,bookmarksnumbered,citecolor=red,urlcolor=red]{hyperref}

\newcommand\BibTeX{{\rmfamily B\kern-.05em \textsc{i\kern-.025em b}\kern-.08em
T\kern-.1667em\lower.7ex\hbox{E}\kern-.125emX}}

\makeatletter
\newif\if@restonecol
\makeatother

\usepackage[algo2e,ruled,lined,titlenumbered,commentsnumbered]{algorithm2e}

\theoremstyle{remark}
\newtheorem*{rem}{Remark} 

\theoremstyle{plain}
\newtheorem{thm}{Theorem}

\newtheorem{cor}{Corollary}

\newcommand{\un}[1]{\ensuremath{\underline{#1}}}
\newcommand{\uu}[1]{\ensuremath{\un{\un{#1}}}}
\newcommand{\dime}{d}

\newcommand{\sig}{\ensuremath{\uu{\sigma}}}
\newcommand{\sigH}{\ensuremath{\sig_H}}

\newcommand\strain[1]{\uu{\varepsilon}\left(#1\right)}
\newcommand{\dep}{\ensuremath{\un{u}}}
\newcommand{\depv}{\ensuremath{\un{v}}}
\newcommand{\depH}{\ensuremath{\dep_H}}

\newcommand{\depvH}{\ensuremath{\depv_H}}

\newcommand{\efdep}{\ensuremath{\mathbf{u}}}
\newcommand{\lam}{\ensuremath{\boldsymbol{\lambda}}}

\newcommand{\eft}{\ensuremath{\mathbf{t}}}

\newcommand\shapef{\varphi}
\newcommand\shapev{\un{\boldsymbol{\shapef}}_H}
\newcommand{\stiff}{\ensuremath{\mathbf{K}}}

\newcommand{\force}{\ensuremath{\mathbf{f}}}

\newcommand{\pa}{\ensuremath{\mathbf{A}}}
\newcommand{\da}{\ensuremath{\mathbf{B}}}

\newcommand{\res}{\ensuremath{\mathbf{r}}}
\newcommand{\bz}{\ensuremath{\mathbf{z}}}

\newcommand{\sigadj}{\ensuremath{\uu{\widetilde{\sigma}}}}
\newcommand{\matid}{\ensuremath{\mathbf{{I}}}}
\newcommand{\depadj}{\ensuremath{\un{\widetilde{u}}}}

\newcommand{\domain}{\ensuremath{\Omega}}
\newcommand{\s}{\ensuremath{^{(s)}}}

\newcommand{\setelem}{\ensuremath{\mathcal{T}}}

\newcommand{\setvertex}{\ensuremath{\mathcal{V}}}
\newcommand\trace{\operatorname{tr}}

\newcommand\hooke{\mathbb{H}}
\newcommand\KA{\ensuremath{\mathrm{KA}}}
\newcommand\KAo{\ensuremath{\KA^0}}
\newcommand\KAoo{\ensuremath{\KA^{00}}}
\newcommand\KAH{\ensuremath{\KA_H}}

\newcommand\SA{\ensuremath{\mathrm{SA}}}

\newcommand\SAt{\ensuremath{\widetilde{\SA}}}
\newcommand\SAtH{\ensuremath{\widetilde{\SA}_H}}

\newcommand{\broken}{\ensuremath{\KA(\bigcup\domain\s)}}

\newcommand\ecr[2]{\ensuremath{{e_{CR_{\ensuremath{#2}}}(\ensuremath{#1})
} } }
\newcommand\ecrc[2]{\ensuremath{{e^2_{CR_{#2}}(#1)}}}
\newcommand\enernorm[2]{\|#1\|_{\hooke^{-1},#2}}

\newcommand{\grad}{\ensuremath{\un{\operatorname{grad}}}}

\title{Strict lower bounds with separation of sources of error in 
non-overlapping domain decomposition methods}
\author{V.~Rey$^1$, P.~Gosselet$^1$, C.~Rey$^2$ \\
$^1$ LMT-Cachan / Ecole Normale Sup\'erieure de Cachan, CNRS, Universit\'e Paris Saclay\\
61, avenue du pr\'esident Wilson, 94235 Cachan, France,\\[5pt]
$^2$ Safran Tech, rue des Jeunes Bois Ch\^ateaufort CS 80112,78772 Magny les Hameaux, France}
\begin{document}


\maketitle

\begin{abstract}
This article deals with the computation of guaranteed lower bounds of the error in the framework of finite element 
(FE) 
and domain decomposition (DD) methods. In addition to a fully parallel computation, the proposed lower bounds separate 
the algebraic error (due to the use of a DD iterative solver) from the discretization error (due to the FE), which 
enables the steering of the iterative solver by the discretization error. These lower bounds are also used to improve 
the goal-oriented error estimation in a substructured context. Assessments on 2D static linear mechanic problems 
illustrate the relevance of the separation of sources of error and the lower bounds' independence from the 
substructuring. We also steer the iterative solver by an objective of precision on a quantity of 
interest. This strategy consists in {a sequence} of {solvings} and takes advantage of 
adaptive
remeshing and recycling of search directions.

\end{abstract}


\vspace{-6pt}

\section{Introduction}
\vspace{-2pt}Virtual testing is a useful tool for engineers to certify structures without resorting to experimental tests. 
However, its massive adaption comes with several challenges. Among others, virtual 
testing requires the capability to solve large problems 
(several millions degrees of freedom) and to warrant the 
quality of the results provided by simulations. To tackle these difficulties, we propose to use domain 
decomposition methods and verification. On the one hand, non-overlapping domain 
decomposition methods \cite{Gos06,Let94,Far94} are well-known techniques that 
enable the {solving} of large 
mechanical problems by exploiting parallel computers' performance. On the other hand, verification provides tools to estimate the distance between the unknown exact solution 
and the computed approximated solution. This distance is the approximation error, it can be estimated by a global energy norm or by local quantities of interest (goal-oriented error estimation). 

This paper is the continuation of papers connecting domain decomposition methods and verification. In \cite{Par10}, the 
authors proposed a parallel error estimator based on the error in constitutive relation \cite{Lad83} in a substructured 
framework. They described a methodology to construct the required admissible fields for error estimation. Those fields 
were rebuilt using quantities naturally processed during the {solving} and preconditioning steps of 
classical domain decomposition algorithms as inputs for classical equilibration techniques used in parallel on each 
subdomain. In \cite{Rey14}, a new parallel error estimator that separates the discretization error (due to the finite 
element method) from the algebraic error (due to the iterative solver) was proposed. This new estimator enables the 
definition of a new 
stopping criterion for the iterative solver, no 
longer defined regardless the discretization, which avoids 
{over-solving}. Finally, in \cite{Rey15}, this work was extended to
goal-oriented error estimation. The exact value of a linear quantity of 
interest defined by an extractor \cite{Par97,Str00,Ohn01} was estimated using 
global error 
estimation of a forward problem and an adjoint problem. It was shown that these 
two 
problems could be solved simultaneously thanks to a block-Krylov algorithm 
\cite{Saa00} steered by an objective on the error on the quantity of interest. 

Upper bounds are the main concern of verification and the literature on lower bounds is scarce. By exploiting the residual equation \cite{Pru99} and 
constructing a 
continuous error estimation, a lower bound of the error can be computed. The 
question of the construction of a continuous error estimation has already been 
adressed in many papers (for instance \cite{Par06, Die03,Gal09}) where the authors 
benefit the computation of an upper bound for the computation of a lower bound.

The objective of this paper is to investigate the computation of a lower bound 
of the error in a substructured framework. In line with the papers 
associating domain decomposition methods and verification, the lower bounds 
are computed in parallel and they separate the two sources of error 
(discretization error and algebraic error). Based on the results demonstrated in \cite{Par06}, the computation of the 
lower bound does not involve significant cost since it exploits the fields computed during the reconstruction of a 
statically admissible stress fields.

The paper is organized as follows. In section \ref{sec:settings}, we define 
the reference problem and recall the principle of the error in constitutive 
relation. We also recall the domain decomposition methods' principles and highlight the 
fields built at each iteration. Finally, we recall the parallel error estimators 
developed in \cite{Par09,Rey14} and give a brief state of the art of the 
computation of continuous fields for lower bounds of the error. In section 
\ref{sec:binf}, two theorems providing lower bounds with and without 
separation of sources of error are demonstrated. The parallel reconstruction of 
fields required to compute these bounds is detailed. We also show how to 
benefit from this global information on the error to better the goal-oriented 
error estimation. In section \ref{sec:assessment}, we apply these lower bounds 
on two-dimensional mechanical structures. We compare the lower bounds provided 
by a sequential approach and primal and dual approaches and also study the 
independence with respect to the substructuring. We illustrate the convergence of the lower 
bounds during the iterations and illustrate the separation of sources. Finally, in order to reach 
an objective of precision on a quantity of interest, we apply an auto-adaptive strategy on one of the structures. In 
this strategy, we use the separation of sources of error to define the stopping criterion for the iterative solver. 
Benefiting the informations from a first {solving}, we process adaptive remeshing to better the FE 
solution and 
lower the error bounds and we recycle the search directions generated (Krylov subspace recycling, see \cite{Rey98, 
Gos13, Ris00}) to speed up further {solvings}. Section \ref{sec:ccl} concludes the paper.

\vspace{-6pt}

\section{Settings}\label{sec:settings}
\vspace{-2pt}
\subsection{Reference problem}
Let $\mathbb{R}^\dime$ represents the physical space. 
Let us consider the static equilibrium  of a (polyhedral) structure which occupies the open domain 
$\domain\subset\mathbb{R}^\dime$ and which is subjected to  given body force 
{$\un{f}\in\mathtt{L}^2(\Omega)$} within $\Omega$,  to given
traction force {$\un{g}\in\mathtt{L}^2(\partial_g\Omega)$} on $\partial_g\Omega$  and to given 
displacement field $\dep_d$ on the  complementary part of 
the boundary (such that $\operatorname{meas}(\partial_u\Omega)\neq0$). We  assume that the structure undergoes  small  
 perturbations  and  that  the  material   is  linear  elastic, characterized by Hooke's  elasticity tensor $\hooke$.  
Let $\dep$ be  the unknown displacement field, $\strain{\dep}$ the symmetric part  of the gradient of $\dep$, $\sig$ 
the Cauchy stress tensor. Let $\omega$ be an open subset of $\domain$.

We introduce two affine subspaces and one positive form:
\begin{itemize}
\item Affine subspace of kinematic admissible fields (KA-fields)
\begin{equation}\label{eq:KA}
  \KA(\omega)=\left\{ \dep\in \left(\mathtt{H}^1(\omega)\right)^\dime,\ \dep = \dep_d \text{ on }\partial\omega\bigcap\partial_u\domain \right\}
\end{equation}
and we note $\KAo(\omega)$ the following linear subspace:
\begin{equation}\label{eq:KAo}
  \KAo(\omega)=\left\{ \dep\in \left(\mathtt{H}^1(\omega)\right)^\dime,\ \dep = 
0 \text{ on }\partial\omega\bigcap\partial_u\domain \right\}
\end{equation}
and $\KAoo(\omega)$ the following linear subspace:
\begin{equation}\label{eq:KAoo}
\KAoo(\omega)=\left\{ \dep\in \left(\mathtt{H}^1(\omega)\right)^\dime,\ \dep 
= 0 \text{ on }\partial\omega\setminus\partial_g\domain \right\}
\end{equation}

\begin{rem}
 Note that if $\omega=\Omega$, $\KAoo(\omega)$ and $\KAo(\omega)$ are 
identical. 
\end{rem}

\item Affine subspace of statically admissible fields (SA-fields)
\begin{multline}\label{eq:SA}
  \SA(\omega)
  =\Bigg\lbrace   \uu{\tau}\in  \left(\mathtt{L}^2(\omega)\right)^{\dime\times \dime}_{\text{sym}}; \
    \forall  \depv \in  \KAoo(\omega),\ \\ \int\limits_\omega
  \uu{\tau}:\strain{\depv}    d\domain    =    \int\limits_\omega   \un{f} \cdot\depv    d\domain +
  \int\limits_{\partial\omega\bigcap\partial_g\domain} \un{g}\cdot\depv dS   \Bigg\rbrace
\end{multline}

\item Error in constitutive relation \cite{Lad83}
\begin{equation}\label{eq:ecr}
  \ecr{\dep,\sig}{\omega}= \enernorm{\sig-\hooke:\strain{\dep}}{\omega}
\end{equation}
where ${\enernorm{\uu{x}}{\omega}}=\displaystyle \sqrt{\int_\omega \left( \uu{x}: {\hooke}^{-1}:\uu{x} \right)d\domain}$
\end{itemize}

The mechanical problem set on $\domain$ can be formulated as:
\begin{equation}\label{eq:refpb}
  \text{Find } \left(\dep_{ex},\sig_{ex}\right)\in\KA(\domain)\times\SA(\domain) \text{ such  that } \ecr{\dep_{ex},\sig_{ex}}{\domain}=0
\end{equation}
The solution to this problem, named ``exact'' solution, exists and is unique.

\begin{rem}
 The formulation \eqref{eq:refpb} is equivalent to the classical 
following formulation:
 
\begin{equation}
 \text{Find }\dep \in \KA(\domain) \text{ such that } \forall  \depv \in  
\KAoo(\Omega) \text{,   }
a(\dep,\depv)=L(\depv)
\end{equation}
with
\begin{equation}\label{eq:def_a_bil}
  a(\dep,\depv)=\int\limits_\Omega \strain{\dep}:\hooke:\strain{\depv}    
d\domain  
\end{equation}
and
\begin{equation}
  L(\depv)= \int\limits_\domain   \un{f} \cdot\depv    d\domain +
  \int\limits_{\partial_g\domain\bigcap\partial\domain} \un{g}\cdot\depv dS
\end{equation}

\end{rem}

\subsubsection{Finite element approximation}
Let us consider a mesh  of ${\domain}$ to which we  associate the 
finite-dimensional  subspace  $\KAH(\domain)$  of $\KA(\domain)$.  
$\setelem$ is the set of elements of the mesh and $\setvertex$ is the 
set of vertexes. The  
classical finite element displacement approximation consists in searching:
\begin{equation}
  \begin{aligned}
    \depH&\in\KAH(\domain)\\
    \sigH&=\hooke:\strain{\depH}   \\
    \int_{\domain}
  \sigH:\strain{\depvH}   d\domain  &=   \int_{\domain} \un{f}\cdot\depvH   
d\domain   +
  \int_{\partial_g\domain} \un{g}\cdot\depvH dS,\qquad \forall\depvH\in 
\KAo_H(\domain)
    \end{aligned}
\end{equation}
Of course the approximation is due to the fact that in most cases 
$\sigH\notin\SA(\domain)$.

After introducing the matrix $\shapev$ of  shape functions which form a basis of 
$\KAH(\domain)$ (extended to Dirichlet degrees of freedom)  and the  vector of  
nodal unknowns  $\efdep$ so that $\depH=\shapev \efdep$, the  classical finite  
element method leads  to the linear
system:
\begin{equation}\label{eq:globalFE}
\begin{pmatrix} \stiff_{rr} & \stiff_{rd} \\ \stiff_{dr} & \stiff_{dd}
\end{pmatrix}\begin{pmatrix}\efdep_r\\ \efdep_d\end{pmatrix} = 
\begin{pmatrix}\force_r\\\force_d \end{pmatrix} +\begin{pmatrix} 0\\\lam_d 
\end{pmatrix}
\end{equation}
where $\stiff$ is  the (symmetric {semi positive definite}) stiffness  matrix
and $\force$ is the vector of generalized forces; Subscript $d$ stands for 
Dirichlet degrees of freedom (where displacements are prescribed) and Subscript 
$r$ represents the remaining degrees of freedom so that unknowns are $\efdep_r$ 
and $\lam_d$ where Vector $\lam_d$ represents the nodal reactions:
\begin{equation}\label{eq:lambdad}
\lam_d^T=
    \int_{\domain}
  \sigH:\strain{{\shapev}_d}   d\domain - \int_{\domain} \un{f}\cdot{\shapev}_d  
d\domain -
  \int_{\partial_g\domain} \un{g}\cdot{\shapev}_d dS
\end{equation}
where ${\shapev}_d$ is the matrix of shape functions restricted to the Dirichlet 
nodes and $\un{n}$ the outer normal vector.

\subsubsection{A posteriori error estimation}
\paragraph{Upper bound of the discretization error}
The estimator we choose is based on the error in constitutive relation, which 
gives a guaranteed estimator for the discretization error.

The fundamental relation is the following (Prager-Synge theorem, see for 
instance \cite{Lad04}):
\begin{equation}\label{eq:erdc}
\begin{split}
&\forall (\hat{\dep},\hat{\sig})\in\KA(\Omega)\times\SA(\Omega),\\ 
&\left\|\strain{\dep_{ex}}-\strain{\hat{\dep}}\right\|_{\hooke,\Omega}^2 + 
\left\|\sig_{ex}-\hat{\sig}\right\|_{\hooke^{-1},\Omega}^2 =  
\ecrc{\hat{\dep},\hat{\sig}}{\Omega}
\end{split}
\end{equation}
We note $\vvvert \depv 
\vvvert_\domain=\left\|\strain{\depv}\right\|_{\hooke,\Omega} $ the energy norm 
of the displacement, and since we can choose $\hat{\dep}=\depH\in\KA(\domain)$, we retain the following upper bound for the error $\un{{e}}_{discr}=\dep_{ex}-\dep_{{H}}$:
\begin{equation}\label{eq:erdc2}
e_{discr}:= \vvvert \un{{e}}_{discr} \vvvert_{\Omega} \leqslant  
\ecr{{\dep}_H,\hat{\sig}}{\Omega}
\end{equation}

 The construction of $\hat{\sig}\in\SA(\domain)$ is a complex problem solved by 
various approaches \cite{Lad83,Par06,Ple11, Rey14bis}. 

The techniques \cite{Lad83, Ple11, Rey14bis} are two-steps procedures. The 
first 
step consist in building a set of equilibrated tractions or works along the 
edges of the elements of the mesh. Each method proposes its own strategy to 
reconstruct such tractions. The second step, common to all methods, is the 
{solving } of Neumann problems on each element using the 
equilibrated tractions as Neumann conditions. 

The technique developed in \cite{Par06} does not require
equilibrated fluxes but only the {solving} local problems on star-patches (a star patch, denoted by 
$\omega_i$, is 
composed of the elements 
sharing the vertex $i$ and corresponds to the support of the shape function 
associated to the vertex $i$). This 
is the reason why the technique is sometimes called the flux-free technique. 

Local problems (on element or star-patch) are usually solved on a space of 
finite dimension which is richer than the finite element space restrained to 
the support of the local problem. The space is enriched either thanks to 
higher degree polynomial shape functions or thanks to mesh refinement (each 
element being divided into several smaller elements).

\paragraph{Lower bound of the discretization error}\label{par:borne_inf_seq}
A lower bound of the true error can be obtained using Cauchy-Schwarz inequality 
in the following residual equation: 
\begin{equation}\label{eq:def_residu_binf}
\begin{aligned}
 \forall \un{w} \in  \KAo(\Omega)\\
\int_{\domain}
  \strain{\dep_{ex}-\depH}:\hooke:\strain{\un{w}}   d\domain&= 
\int\limits_\Omega   \un{f} \cdot\un{w}    d\domain +
  \int\limits_{\partial_g\domain\bigcap\partial\Omega} \un{g}\cdot\un{w} dS  - 
 \int_{\domain}\sigH:\strain{\un{w}}   d\domain \\
 & := R_H(\un{w})
 \end{aligned}
\end{equation}

Therefore, every displacement field $\un{w} \in 
\KAo(\Omega)\setminus\{\un{0}\}$ can 
be used to obtain the following strict lower bound \cite{Pru99}:
\begin{equation}\label{eq:binf}
\vvvert \dep_{ex} - \dep_H \vvvert_{\Omega} \geq \frac 
{|R_H(\un{w})|}{\vvvert\un{w}\vvvert_{\Omega}}
\end{equation}

The accuracy of the lower bound depends on the quality of the continuous field 
$\un{w}$, which is often called continuous error estimate. Indeed, the lower bound equals the true error for the 
continuous field $\un{w}=\dep_{ex} - \dep_H $. 
The construction of $\un{w}$ was mainly studied in \cite{Die03, Par06, Gal09} of which we recall the main results.

In \cite{Die03}, the error is estimated thanks to an implicit residual-based 
error estimator with local {solvings} on elements. A continuous field $\un{w} 
\in 
\KAo(\Omega)\setminus\{\un{0}\}$ is constructed from the element estimators by averaging on the edges.

 In \cite{Par06}, local problems on star-patches are solved to compute an upper 
bound of the error : 
\begin{equation}
\begin{aligned}
  \text{Find }\underline{e}^i \in \KAo(\omega_i) \text{ such that } \forall 
 \depv \in  \KAo(\omega_i)\\
 a(\underline{e}^i, \depv)=R_H(\varphi_H^i \depv)
 \end{aligned}
\end{equation}
where $\varphi_H^i$ is the shape function associated to the central node i of 
the star-patch.
The previous problem is solved on a space $\KAo_h(\omega_i)$ richer than $\KAo_H(\omega_i)$.

The proposed continuous displacement field is:
\begin{equation}
 \un{w}= \Pi_h(\sum_{i\in\setvertex} \varphi_H^i  \underline{e}^i )
\end{equation}
where  $\Pi_h$ is the projector on the space $\KAo_h(\omega_i)$.
Discontinuities between star-patches vanish thanks to the multiplication by 
$\varphi_H^i$.
The projector eases the computation of $R_H(\un{w})$. 

Note that in the same article, an enhanced estimate is proposed to better the 
lower bound : 
\begin{equation}\label{eq:binf_amelioree}
\vvvert \dep_{ex} - \dep_H \vvvert^2_{\Omega} \geq \frac 
{R_H(\un{w})^2}{\vvvert\un{w}\vvvert^2_{\Omega} - 
\vvvert\un{e}_G\vvvert^2_{\Omega}}
\end{equation}
where $\un{e}_G$ is the solution of the following global problem: 
\begin{equation}
\begin{aligned}
  \text{Find }\underline{e}_G \in \KAo_H(\Omega) \text{ such that } 
\forall \depv \in  \KAo_h(\omega_i)\\
 a(\underline{e}_G, \depv)=- a(\un{w}, \depv)
 \end{aligned}
\end{equation}

In  \cite{Gal09}, the statically admissible stress field is built from a 
displacement field which is the sum of the solutions of local problems on 
star-patches with homogeneous boundary Dirichlet conditions: 
\begin{equation}
\begin{aligned}
  \text{Find }\un{w}^i  \in\KA^{0,\omega_i}(\omega_i) \text{ such that }   \forall \depv 
\in\KA^{0,\omega_i}(\omega_i)\\
\int_{\omega_i}\strain{\un{w}^i}:\hooke :  \strain{\depv} d\Omega= 
\int_{\omega_i}  
(\varphi_i \un{f} -\sig_H \grad(\varphi_i))\depv d\Omega - \int_{\omega_i}  
\varphi_i 
\sig_H: \strain{\depv} d\Omega
 \end{aligned}
\end{equation}
where $\KA^{0,\omega_i}$ is the space of continuous displacement fields that equal to zero on the 
boundary of the star-patch $\omega_i$.
The continuous field $\un{w} \in \KAo(\Omega)$ is the sum of the solutions of 
the previous problem:
\begin{equation}
 \un{w}=\sum_{i\in\setvertex} \un{w}^i
\end{equation}

To conclude this brief review, there exist various techniques to construct $\un{w}$. They always take advantage of the 
field computed during the estimation of an upper bound so that the extra-cost is very limited.

\subsubsection{Substructured formulation}
 Let us consider a decomposition of domain $\Omega$ in $N_{sd}$ regular open subsets $(\Omega^{(s)})_s$ such that 
$\Omega^{(s)}\bigcap \Omega^{(s')}=\emptyset$ for $s\neq s'$ and $\bar{\Omega}=\bigcup_s \bar{\Omega}^{(s)}$. We note $\partial_g\domain\s= \partial\Omega\s \bigcap \partial_g\Omega$ the Neumann border of subdomains.

The mechanical problem on the substructured configuration writes :
\begin{equation}
  \forall s \left\{
   \begin{aligned}
&\dep^{(s)} \in \KA(\Omega^{(s)})  \\
& \sig^{(s)} \in \SA(\Omega^{(s)}) \\
&  \ecr{\dep^{(s)},\sig^{(s)}}{\Omega^{(s)}}=0
\end{aligned}  \right.
\end{equation}

and
\begin{equation}
 {\forall (s,s') \text{ such that } \Omega^{(s)}  \text{ and }  
\Omega^{(s')} \text{ are adjacent }}\left\{
    \begin{aligned}
   & \trace(\dep^{(s)})=\trace(\dep^{(s')}) \text{ on } \Gamma^{(s,s')}\\
 &  \sig^{(s)} \cdot\underline{n}^{(s)} +\sig^{(s')}\cdot \underline{n}^{(s')} =\underline{0} \text{ on } 
\Gamma^{(s,s')}
    \end{aligned} 
  \right.
\end{equation}
The set of fields $\dep$ defined on $\domain$ such that $\dep_{|\domain\s}\in\KA(\Omega^{(s)})$ without interface 
continuity is a broken space which we note \broken.

\subsubsection{Finite element approximation for the substructured problem}
We assume that the mesh of  $\bar{\domain}$ and the substructuring are conforming. This hypothesis implies that each 
element only belongs to one subdomain and nodes are matching on the interfaces. {For each subdomain, 
let us denote by the subscript b the degrees of freedom on the boundary of the subdomain and by the subscript i the 
degrees of fredom inside the subdomain.}

Let  $\eft^{(s)}$ be the discrete trace operator on the interface. 
{$\eft^{(s)}$ enables to cast degrees of freedom from a complete subdomain to its interface. Using an 
adapted ordering, we have
\begin{equation}\label{eq:discrete_trace}
 \eft^{(s)}=\begin{pmatrix} \matid_{bb} & \mathbf{0}_{bi} 
\end{pmatrix}
\end{equation}
Therefore 
${\eft\s}^T$ is the extension by zero operator: it extends data supported by the boundary to the whole subdomain.}

Let us introduce the unknown nodal reaction on the interface $\lam\s$, the equilibrium of each subdomain writes:
\begin{equation}\label{eq:efddeq}
\stiff\s \efdep\s = \force\s + {\eft\s}^T \lam\s
\end{equation}

Let $(\pa\s)$ and $(\da\s)$ be the primal and dual assembly operator. {Those operators are signed boolean operators. 
The dual operator $(\da\s)$ enables to express the continuity of displacements and the primal operator $(\pa\s)$ 
enables to express the mechanical equilibrium of interface. 
Their number of columns is equal to the number of boundary degrees of freedom. $\pa\s$ injects the boundary degrees of freedom of $\Omega\s$ in the global interface.
Thus the number of rows of $\pa\s$ is equal to the number of degrees of freedom on the global interface. 
The number of rows of $\da\s$ is equal to the number of connections between pairs of neighboring degrees of freedom. }

{In the case of two subdomains, we have $\sum_s\da\s \eft\s\efdep\s = \eft^{(1)}\efdep^{(1)} -\eft^{(2)}\efdep^{(2)} $ 
and $\sum_s\pa\s \lam\s~=~\lam^{(1)}+\lam^{(2)} $. For more details on the assembly operators, the reader can refer 
to~\cite{Gos06}.}

The discrete counterpart of the interface 
admissibility equations is:
\begin{equation}\label{eq:efintadmiss}
\left\{\begin{aligned}
\sum_s\da\s \eft\s\efdep\s &=0 \\
\sum_s\pa\s \lam\s &=0 \\
\end{aligned}\right.
\end{equation}
Equations \eqref{eq:efddeq} and \eqref{eq:efintadmiss} form the discrete substructured 
system, which is equivalent to the global problem \eqref{eq:globalFE}.

\subsubsection{Domain decomposition solvers and admissible fields}\label{sec:definitions_champs_iter}
Domain decomposition solvers are well described in many papers (see for 
instance \cite{Gos06} and the associated bibliography). The principle is to condense the global problem on the 
interface to create a smaller problem. In classical algorithms such as BDD \cite{Let94} and FETI \cite{Far94}, this 
new interface problem is solved iteratively thanks to a projected preconditioned conjugate gradient. Each iteration 
implies two parallel {solvings} on subdomains with two dual operators (one for the 
preconditioning step, one for the direct step) so that local problems with Neumann boundary conditions and local 
problems with Dirichlet boundary conditions are alternatively solved. In \cite{Par10} it was proved that the following 
fields could be processed at no extra cost:
\begin{itemize}
\item $(\dep_D\s)_s\in \KA(\domain)$: {displacement field which results from a Dirichlet problem and 
which is 
thus globally admissible}
\item $(\lam_N\s)_s$ :  nodal reactions which are balanced at 
the interface.  
\item $(\dep_N\s)_s\in 
\broken$: { displacement field which results from a Neumann problem and which is not 
globally admissible}
\item $\sig_N\s $: the stress field associated to $\dep_N\s$ ($\sig_N\s = \hooke:\strain{\dep_N\s}$). It 
can be {used} (with 
additional input $\lam_N\s$) to build in parallel stress fields $ \hat{\sig}_N\s$ which are statically admissible 
$\hat{\sig}_N=(\hat{\sig}_N\s)_s\in\SA(\Omega)$ using dedicated methods such 
as \cite{Lad83,Par06,Ple11, Rey14bis}. 
\end{itemize}

In \cite{Rey14} we proved the following result where $\alpha$ (denoted $\sqrt{\res^T\bz}$ in \cite{Rey14}) is the 
preconditioner-norm of the residual, a quantity that is actually computed by the solver:
\begin{equation}\label{eq:lemma2}
\alpha := \lVert \un{u}_N-\un{u}_D  \rVert_{\hooke,\Omega}  = \sqrt{\res^T\bz}
\end{equation}
\subsection{A posteriori upper bound of the error in substructured context}\label{sec:dd_error_estimation}
In \cite{Par10}, a first parallel error estimator in substructured context was 
introduced. It is based on the error in constitutive relation and reads :
\begin{equation}\label{eq:estimgus}
 \vvvert  \dep_{ex} -\dep_D \vvvert_\domain = \sqrt{\sum_s \vvvert  \dep_{ex}\s 
-\dep_D\s \vvvert_{\domain\s}^2}\leqslant 
\sqrt{\sum_s\ecrc{\dep_D\s,\hat{\sig}_N\s}{\Omega\s}}
\end{equation}

 In \cite{Rey14}, we showed that this estimator mixes two different sources of 
error : the discretization error which is inherent to the use of the finite 
element method and the algebraic error which is due to the use of 
an iterative solver which would not exist for a direct solver. The algebraic error monitors the convergence of the 
solver and can be made as small as wished. Therefore, a 
second parallel error estimator was proposed in \cite{Rey14}:

\begin{equation}\label{eq:erralg}
 \vvvert  \un{u}_{ex}-\un{u}_N  \vvvert _{\Omega} \leq  \alpha + 
\sqrt{\sum_s\ecrc{\dep_N\s,\hat{\sig}_N\s}{\Omega\s}}
\end{equation}
\begin{equation}\label{eq:erralg2}
 \vvvert  \un{u}_{ex}-\un{u}_D  \vvvert _{\Omega} \leq  \alpha + 
\sqrt{\sum_s\ecrc{\dep_N\s,\hat{\sig}_N\s}{\Omega\s}}
\end{equation}

This estimator separates the two sources of error. When the solver {has converged} the two 
displacements fields 
$\un{u}_N$ and $\un{u}_D$ are identical and equal to $\un{u}_H$ (the algebraic 
error $\alpha$ is very close to zero). $\sqrt{\sum_s\ecrc{\dep_N\s,\hat{\sig}_N\s}{\Omega\s}}$ is the estimation 
of the discretization error. 
As a consequence, at convergence, the estimators 
\eqref{eq:estimgus}, \eqref{eq:erralg} and \eqref{eq:erralg2} are identical.

\vspace{-6pt}
\section{Lower bound of the error in substructured context}\label{sec:binf}
\vspace{-2pt}
In this section, we extend sequential results to demonstrate guaranteed lower bounds of the error in substructured 
context. Moreover, we prove a theorem that enables the separation of sources of error in the lower bound. We also 
develop the methodology to build a continuous error estimate from parallel error estimation procedure. Finally, we 
extend those results to goal-oriented error estimation.

\subsection{A first lower bound of the error}

\begin{thm}\label{thm:borne_inf_DD}
 Let $\un{u}_{ex} \in \KA(\Omega)$ be the exact solution, 
$(\un{u}_D\s)_s\in\KA(\Omega)$ the displacement field defined in~\ref{sec:definitions_champs_iter} and  
$\un{w}\in\KAo(\Omega)\setminus\{\un{0}\}$ then
\begin{equation}
  \vvvert  \un{u}_{ex}-\un{u}_D  
\vvvert_{\Omega}\geq\frac{|{R}_D(\un{w})|}{\sqrt{\sum_s\vvvert  
\un{w}\s \vvvert_{\Omega\s}^2}}
\end{equation}
with
\begin{equation}
\begin{aligned}
{R}_D(\un{w})&=\sum_s R_D\s(\un{w}\s)\\
\end{aligned}
\end{equation}
\begin{equation}
\begin{aligned}
{R}_D\s(\un{w}\s)&:=\int\limits_{\domain\s}   \un{f} 
\cdot\un{w}\s    
d\domain
+ \int\limits_{\partial_g\domain\s} \un{g}\cdot\un{w}\s 
dS - \int\limits_{\Omega\s} \strain{\dep_D\s}:\hooke:\strain{\un{w}\s}    
d{\domain\s} 
\end{aligned}
\end{equation}
\end{thm}
\begin{proof}
This property is the direct application of  \eqref{eq:binf} where we replace
the displacement field $\dep_H \in \KA(\Omega)$ by 
$(\un{u}_D\s)_s\in\KA(\Omega)$.
The residual $R_D(\un{w})$ can be rewritten:
\begin{equation}
\begin{aligned}
R_D(\un{w})&=\sum_s R_D\s(\un{w}\s)\\
&=\sum_s \left(\int\limits_{\domain\s}   \un{f} 
\cdot\un{w}\s    
d\domain+ \int\limits_{\partial_g\domain\s} 
\un{g}\cdot\un{w}\s 
dS - \int\limits_{\Omega\s} \strain{\dep_D\s}:\hooke:\strain{\un{w}\s}    
d{\domain\s}  \right) \\
&=L(\un{w})-a(\dep_D,\un{w})
\end{aligned}
\end{equation}
\end{proof}

In practice, we choose $\un{w}\s\in\KAoo(\Omega\s)\setminus\{\un{0}\}\subset\KAo(\Omega)\setminus\{\un{0}\}$, which 
corresponds to imposing the nullity along the interface and which is {inexpensive} since 
it does not imply exchanges between subdomains. In subsection~\ref{sec:construction_w}, we will give details about the 
computation of $\un{w}$. Therefore the computation of a lower bound is as parallel 
as for the upper bound. In the assessments in section~\ref{sec:assessment}, we will 
verify that this lower bound is as accurate as the one obtained in the sequential 
context and that the quality neither depends on the approach (primal or dual) 
nor on the substructuring.
Moreover, this lower bound is computable whatever the state of the iterative 
solver (converged or not).

\subsection{Lower bound with separation of sources of error }
Continuing the philosophy of separating the sources of error as detailed in 
\cite{Rey14}, we propose a second lower bound : 

\begin{thm}\label{thm:borne_inf_dd_separation_1}
Let $\un{u}_{ex} \in \KA(\Omega)$ be the exact solution, 
$(\un{u}_D\s)_s\in\KA(\Omega)$ and 
$(\un{u}_N\s)_s\in\broken$ the displacement fields defined in~\ref{sec:definitions_champs_iter} 
and $\un{w}\in\KAo(\Omega)\setminus\{\un{0}\}$ , then
 \begin{equation}\label{eq:binf_separation1a}
  \vvvert  \un{u}_{ex}-\un{u}_D  \vvvert 
_{\Omega}\geq\left|\frac{|{R}_N(\un{w})|}{ \sqrt{\sum_s\vvvert  
\un{w}\s \vvvert_{\Omega\s}^2} }
 -\frac{|a(\un{u}_D-\un{u}_N,\un{w})|}{ \sqrt{\sum_s\vvvert  
\un{w}\s \vvvert_{\Omega\s}^2}}\right|
\end{equation}
which leads to the coarser bound
 \begin{equation}\label{eq:binf_separation1b}
  \vvvert  \un{u}_{ex}-\un{u}_D  \vvvert 
_{\Omega}
\geq\frac{|R_N(\un{w})|}{ \sqrt{\sum_s\vvvert  
\un{w}\s \vvvert_{\Omega\s}^2}} -\alpha
\end{equation}
with 
\begin{equation}
\begin{aligned}
{R}_N(\un{w})&=\sum_s R_N\s(\un{w}\s)\\
\end{aligned}
\end{equation}
\begin{equation}
\begin{aligned}
{R}_N\s(\un{w}\s)&:=\int\limits_{\domain\s}   \un{f} 
\cdot\un{w}\s    
d\domain
+ \int\limits_{\partial_g\domain\s} \un{g}\cdot\un{w}\s 
dS - \int\limits_{\Omega\s} \strain{\dep_N\s}:\hooke:\strain{\un{w}\s}    
d{\domain\s} 
\end{aligned}
\end{equation}

\end{thm}

\begin{proof}
 
The proof of the first inequality is based on theorem~\ref{thm:borne_inf_DD} and on the triangle inequality : 
 \begin{equation}
\begin{aligned}
 \vvvert  \un{u}_{ex}-\un{u}_D  \vvvert 
_{\Omega}&\geq\frac{|{R}_D(\un{w})|}{ \sqrt{\sum_s\vvvert  
\un{w}\s \vvvert_{\Omega\s}^2}}\\
 \vvvert  \un{u}_{ex}-\un{u}_D  \vvvert 
_{\Omega}&\geq\frac{|L(\un{w})-a(\un{u}_D,\un{w})|}{ \sqrt{\sum_s\vvvert  
\un{w}\s \vvvert_{\Omega\s}^2}}\\
&\geq\frac{|L(\un{w})-a(\un{u}_N,\un{w})-a(\un{u}_D-\un{u}_N,\un{w})|}{\sqrt{
\sum_s\vvvert  
\un{w}\s \vvvert_{\Omega\s}^2}}\\
&\geq\left|\frac{|{R}_N(\un{w})|}{ \sqrt{\sum_s\vvvert  
\un{w}\s \vvvert_{\Omega\s}^2} }
 -\frac{|a(\un{u}_D-\un{u}_N,\un{w})|}{ \sqrt{\sum_s\vvvert  
\un{w}\s \vvvert_{\Omega\s}^2}}\right|\\
\end{aligned}
\end{equation}
which proves \eqref{eq:binf_separation1a}. \eqref{eq:binf_separation1b} is simply based on the remark that
 \begin{equation}
 \left|\frac{|R_N(\un{w})|}{ \sqrt{\sum_s\vvvert  
\un{w}\s \vvvert_{\Omega\s}^2} }
 -\frac{|a(\un{u}_D-\un{u}_N,\un{w})|}{ \sqrt{\sum_s\vvvert  
\un{w}\s 
\vvvert_{\Omega\s}^2}}\right|\geq\frac{|R_N(\un{w})|}{ 
\sqrt{\sum_s\vvvert  
\un{w}\s \vvvert_{\Omega\s}^2} }
 -\frac{|a(\un{u}_D-\un{u}_N,\un{w})|}{ \sqrt{\sum_s\vvvert  
\un{w}\s \vvvert_{\Omega\s}^2}}
\end{equation}
and using twice the Cauchy-Schwarz inequality, we have : 
 \begin{equation}
\begin{aligned}
|a(\un{u}_D-\un{u}_N,\un{w})|&= |\sum_s \int_{\Omega\s} 
\strain{\un{u}_D\s-\un{u}_N\s}: \hooke:\strain{\un{w}\s}|\\
&\leq \sum_s  \vvvert  \un{u}_{N}\s-\un{u}_D\s  \vvvert_{\Omega\s} \vvvert  
\un{w}\s \vvvert_{\Omega\s}\\
&\leq   \sqrt{\sum_s\vvvert  \un{u}_{N}\s-\un{u}_D\s  \vvvert^2_{\Omega\s}} 
\sqrt{\sum_s\vvvert  
\un{w}\s \vvvert^2_{\Omega\s}}\\
\end{aligned}
\end{equation}
Finally, using the equality \eqref{eq:lemma2}: 
 \begin{equation}
  \vvvert  \un{u}_{ex}-\un{u}_D  \vvvert 
_{\Omega}\geq\frac{|R_N(\un{w})|}{ \sqrt{\sum_s\vvvert  
\un{w}\s \vvvert_{\Omega\s}^2}} -\alpha
\end{equation}

\end{proof}

 As said earlier, the term $\alpha$ is a measure of the residual so it is purely 
algebraic whereas the first term of the inequality is mainly driven  by the discretization error. 
{During the first 
iterations, the second lower bound is not 
accurate because the algebraic error 
prevails so that $\frac{|L(\un{w})-a(\un{u}_N,\un{w})|}{ \sqrt{\sum_s\vvvert  
\un{w}\s \vvvert_{\Omega\s}^2}} -\alpha$ is negative }and it is a trivial lower bound of the positive true error  
$\vvvert 
\un{u}_{ex}-\un{u}_D  \vvvert_{\Omega}$.
When the solver reaches convergence, the three lower bounds in 
theorems~\ref{thm:borne_inf_DD} and~\ref{thm:borne_inf_dd_separation_1} are 
identical.

\subsection{Reconstruction of admissible field $\un{w}$}\label{sec:construction_w}

The upper bounds of the error \eqref{eq:erralg} or \eqref{eq:erralg2} require the 
construction of a statically admissible field $(\hat{\sig}_N\s)_s$. In case the 
flux-free technique is chosen, it is 
possible to construct a continuous field $\un{w}\s \in \KAoo(\Omega\s)$ using 
the methodology developed in \cite{Par06} for each subdomain in parallel. As a 
consequence, 
\begin{equation}
\left((\Pi_h(\sum_{i\in\setvertex} \phi_H^i  \underline{e}^i ))\s\right)_s\in\broken
\end{equation}
In order to have $\un{w}\s\in \KAoo(\Omega\s)$, we choose not to sum the 
contributions 
from the nodes located on the interface of the subdomain $(s)$. They will be 
denoted as $\setvertex_\Gamma$. $\un{w}\s$ is defined by :
\begin{equation}
\un{w}\s=(\Pi_h(\sum_{i\in\setvertex\setminus\setvertex_\Gamma} \varphi_H^i 
\underline{e}^i ))\s 
\end{equation}
Therefore $\un{w}\s\in \KAoo(\Omega \s)$ and
$\un{w}=(\un{w}\s)_s\in \KAo(\Omega)$.

\begin{rem}
Using the subtle trick in \cite{Par06}, the computation of the 
discretized fields is eased. Indeed:
\begin{equation}
\mathbf{w}\s= \sum_{i\in\setvertex\setminus\setvertex_\Gamma} 
{\mathbf{\boldsymbol{\varphi}_H^i }} \odot \mathbf{{e}^i }\s
\end{equation}
and
\begin{equation}
\mathbf{R_D}\s(\mathbf{w}\s)\\=\sum_{i\in\setvertex-\setvertex_\Gamma} 
{\mathbf{\boldsymbol{\varphi}_H^i }} \odot \mathbf{R_D}\s(\mathbf{{e}^i }\s)
\end{equation}
where $\mathbf{\boldsymbol{\varphi}_H^i } $ gathers the nodal values of the 
shape function $\varphi_H^i$ projected on the richer space used 
to solve the star-patch problem whose $\mathbf{{e}^i }\s$ is the discretized 
solution and where $\odot$ represents the term by term multiplication.
 
\end{rem}

\begin{rem}
If the method chosen to construct the admissible field is based on {elements problems }
\cite{Die03}, it is always possible to construct a displacement 
field $\un{w}\s \in \KAoo(\Omega\s)$ by computing the mean value along the 
edges inside the subdomains and imposing zero along the interfaces between 
subdomains.
\end{rem}

\subsection{Goal-oriented error estimation}
Goal-oriented error estimation offers the possibility to have upper and lower bounds of the unknown exact value of 
a quantity of interest. Among various techniques, extractors (see \cite{Bec96} for instance) are the most 
common tools to define linear quantities of interest. They lead to the {definition and the solving} of 
an adjoint problem.
\subsubsection{Definition of the linear quantity of interest and of the adjoint 
problem }
Let $\widetilde{L}$ 
be the linear functional defining the quantity of interest $I$: 
\begin{equation}
I=\widetilde{L}(\dep)=\int_\Omega(\sig_{\Sigma} : \strain{\dep} + 
\un{f}_{\Sigma} \dep 
)d\Omega
\end{equation}
where $\sig_{\Sigma}$ and $\un{f}_{\Sigma}$ are extractors.

We introduce the affine subspace of statically admissible fields 
(adjoint SA-fields) for the 
adjoint problem:
\begin{multline}\label{eq:SA_adjoint}
  \SAt(\omega)
  =\Bigg\lbrace   \uu{\tau}\in  \left(\mathtt{L}^2(\omega)\right)^{\dime\times 
\dime}_{\text{sym}}; \
    \forall  \depv \in  \KAoo(\omega),\ \int\limits_\omega
  \uu{\tau}:\strain{\depv}    d\omega   =    \tilde{L} (\depv)   \Bigg\rbrace
\end{multline}

The adjoint problem set on $\domain$ can be formulated as:
\begin{equation}\label{eq:refpbad}
  \text{Find } \left(\depadj_{ex},\sigadj_{ex} 
\right)\in\KAo(\domain)\times\SAt(\domain) \text{ such  that } 
\ecr{\depadj_{ex},\sigadj_{ex}}{\domain}=0
\end{equation}
The solution to this problem, named exact solution, exists and is unique.

\begin{rem}
 The formulation \eqref{eq:refpbad} is equivalent to the classical 
following formulation:
 
\begin{equation}
 \text{Find }\depadj_{ex} \in \KAo(\domain) \text{ such that } \forall  \depv \in  
\KAo(\Omega) \text{,   }
a(\depadj_{ex},\depv)=\widetilde{L}(\depv)
\end{equation}
where a is the classical bilinear form \eqref{eq:def_a_bil}.

\end{rem}

The adjoint problem is usually solved with the finite element method. The 
mesh can differ from the one used for the forward problem. The 
approximated adjoint
displacement is ${\depadj}_{\widetilde{H}}$.

The discretization error for the adjoint problem is 
\begin{equation}
 e_{discr}= \vvvert \un{\widetilde{e}}_{discr} \vvvert_{\Omega} = \vvvert 
\depadj_{ex}-\depadj_{\widetilde{H}}\vvvert_{\Omega}
\end{equation}

\subsubsection{Error estimation of quantities of interest}
As said earlier the adjoint problem is meant to extract one quantity of interest 
in the forward problem. Let $I_{ex}=\tilde{L}(\dep_{ex})$ be the unknown exact 
value of the quantity of interest. $I_H=\tilde{L}({\dep}_D)$ is an approximation of 
this quantity of interest.

A bounding of the exact value of the quantity of interest $I_{ex}$ is 
\cite{Lad06, Lad08}:
\begin{equation}\label{eq:majoration_IHH2}
\arrowvert I_{ex}-I_H-I_{HH2} \arrowvert\leq 
\frac{1}{2}\ecr{\depH,\hat{\sig}_H}{\Omega} 
\ecr{{\depadj}_{\widetilde{H}},\hat{\sigadj}_{\widetilde{H}}}{\Omega}
\end{equation}
where
\begin{equation}\label{eq:def_IHH2}
 I_{HH2}=  \int\limits_\Omega  
\frac{1}{2}(\hat{\sigadj}_{\widetilde{H}}+\hooke:\strain{\depadj_{\widetilde{H}}
}
) :\hooke^{-1} :(\hat{\sig}_H-\hooke:\strain{\depH})   d\Omega   
\end{equation}
and where $\hat{\sigadj}_{\widetilde{H}} \in   \SAtH(\Omega)$.

The error on the quantity of interest can also be estimated using the 
parallelogram identity \cite{Pru99}: 
\begin{equation}\label{eq:id_parallelogramme}
\left\{\begin{aligned}
 I_{ex}-I_H=\widetilde{L}(\un{e}_{discr})&= 
a(\un{e}_{discr},\un{\widetilde{e}}_{discr}) \\
 &= a(\kappa\un{e}_{discr},\frac{1}{\kappa}\un{\widetilde{e}}_{discr}) \\
 &= \frac{1}{4} [\vvvert{{\kappa\un{e}_{discr}+ 
\frac{1}{\kappa}\un{\widetilde{e}}_{discr} }}\vvvert^2_{\Omega}    
-\vvvert{\kappa\un{e}_{discr}- 
\frac{1}{\kappa}\un{\widetilde{e}}_{discr}}\vvvert^2_{\Omega}]\\
 \end{aligned}\right.
\end{equation}
where $\kappa$ is a scalar parameter whose optimal value is: 
\begin{equation}\label{eq:kappa}
\kappa=\frac{\ecr{{\depadj}_{H},\hat{\sigadj}_{H}}{
\Omega}}{\ecr{\depH,\hat{\sig}_H}{\Omega}} 
\end{equation} 
This optimal value minimizes the difference and thus improves the quality 
of the bounding. Introducing upper and lower bounds: 
\begin{equation}\label{eq:boundsqti}
\begin{aligned}
 \beta^+_{inf} &\leq \vvvert\kappa\un{e}_{discr}+ 
\frac{1}{\kappa}\un{\widetilde{e}}_{discr} \vvvert_{\Omega}^2 \leq\beta^+_{sup} \\
  \beta^-_{inf} &\leq \vvvert\kappa\un{e}_{discr}- 
\frac{1}{\kappa}\un{\widetilde{e}}_{discr}\vvvert_{\Omega}^2\leq\beta^-_{sup} \end{aligned}
\end{equation}
it is possible to obtain lower and upper bounds of the error on the 
quantity of interest :
\begin{equation}\label{eq:id_parallelogramme_appli}
\frac{1}{4} \beta^+_{inf} -\frac{1}{4}  \beta^-_{sup}   \leq 
I_{ex}-I_H  \leq \frac{1}{4} \beta^+_{sup} -\frac{1}{4}  
\beta^-_{inf} 
\end{equation}

\subsubsection{Application to the substructured context}\label{sec:definitions_champs_iter_adjoint}

Let us suppose that the forward and adjoint problems are solved on the same mesh and on the same substructuring. Since 
the two problems share the same stiffness matrices on every subdomain, they can be solved 
simultaneously using a block algorithm. Following the same methodology described in 
section~\ref{sec:definitions_champs_iter}, one can compute the following admissible fields for the adjoint problem:
\begin{itemize}
\item $(\depadj_D\s)_s\in \KAo(\domain)$: {displacement field which results from a Dirichlet problem and 
which is 
thus globally admissible}
\item  $\widetilde{\lam}_N\s$ :  nodal reactions which are balanced at 
the interface.  
\item$(\depadj_N\s)_s\in\broken$:  { displacement field which results from a Neumann problem and which 
is 
not 
globally admissible}
\item $\sigadj_N\s $: stress field associated to $\depadj_N\s$ ($\sigadj_N\s = \hooke:\strain{\depadj_N\s}$). It 
can be {used} (with 
additional input $\widetilde{\lam}_N\s$) to build in parallel stress fields $ \hat{\sigadj}_N\s$ which 
are statically admissible 
$\hat{\sigadj}_N=(\hat{\sigadj}_N\s)_s\in\SAt(\Omega)$. 
\end{itemize}
We also have the equality that expresses the distance between the Neumann and Dirichlet displacement fields in terms of 
the algebraic residual $\widetilde{\alpha}$ of the adjoint problem :
\begin{equation}\label{eq:lemma2_adj}
\lVert \depadj_N-\depadj_D \rVert_{\hooke,\Omega} = \widetilde{\alpha}
\end{equation}

For more details about the computation of those fields, the reader can refer to~\cite{Rey15}.

The upper bounds of the global error presented in section~\ref{sec:dd_error_estimation} and the lower bound in the 
theorem~\ref{thm:borne_inf_DD} can be applied on the adjoint problem.
\paragraph{Upper and lower bounds for goal-oriented error estimation}

We demonstrate two properties that give upper and lower bounds of the terms in the parallelogram identity. The 
properties are merely the application of the 
results on goal-oriented error estimation \cite{Lad06, Lad08, Par06} into a 
substructured context.

\begin{cor}\label{prop:S+}
Using notations of paragraph~\ref{sec:definitions_champs_iter} and~\ref{sec:definitions_champs_iter_adjoint}
 \begin{equation}
 \begin{aligned}
  &\beta^+_{inf}\leq {\vvvert\kappa\un{e}_{discr}+ 
\frac{1}{\kappa}\un{\widetilde{e}}_{discr} \vvvert_{\Omega}}^2 \leq  
\beta^+_{sup}\\
  &\beta^-_{inf}\leq {\vvvert\kappa\un{e}_{discr}- 
\frac{1}{\kappa}\un{\widetilde{e}}_{discr} \vvvert_{\Omega}}^2 \leq  
\beta^-_{sup}
\end{aligned}
 \end{equation}
with 
\begin{equation}\label{eq:identite_para_dd+}\left\{
\begin{aligned}
 \beta^+_{sup}&=2 \sqrt{\sum_s \ecr{\dep_D\s,\hat{\sig}_N\s}{\Omega\s} ^2}
 \sqrt{\sum_s 
\ecr{{\depadj}_{D}\s,\hat{\sigadj}_{N}\s}{\Omega\s}^2} \\
&+2 \sum_s \int_{\domain\s} 
(\hat{\sigadj}\s_{N}-\hooke\strain{\depadj\s_{D}}):\hooke^{
-1}:(\hat{\sig}\s_N-\hooke:\strain{\dep_D\s})  d\domain\\
\beta^+_{inf}&= \frac{(\kappa {R_D}( 
\un{z}^+) + \frac{1}{\kappa} \widetilde{{R}}_D(
\un{z}^+))^2}{{\sum_s \vvvert(\un{z}^+)\s \vvvert^2_{\Omega\s}}}\\
 &\text{ with }(\underline{z}^+)\s=\kappa \un{w}\s + 
\frac{1}{\kappa}\un{\widetilde{w}}\s
\end{aligned}\right.
\end{equation}
and
\begin{equation}\label{eq:identite_para_dd-}\left\{
\begin{aligned}
 \beta^-_{sup}&=2 \sqrt{\sum_s \ecr{\dep_D\s,\hat{\sig}_N\s}{\Omega\s} ^2}
 \sqrt{\sum_s 
\ecr{{\depadj}_{D}\s,\hat{\sigadj}_{N}\s}{\Omega\s}^2} \\
&-2 \sum_s \int_{\domain\s} 
(\hat{\sigadj}\s_{N}-\hooke\strain{\depadj\s_{D}}):\hooke^{
-1}:(\hat{\sig}\s_N-\hooke:\strain{\dep_D\s})  d\domain\\
\beta^-_{inf}&= \frac{(\kappa {R_D}( 
\un{z}^-) - \frac{1}{\kappa} \widetilde{{R}}_D(
\un{z}^-))^2}{{\sum_s \vvvert(\un{z}^-)\s \vvvert^2_{\Omega\s}}}\\
 &\text{ with }(\underline{z}^-)\s=\kappa \un{w}\s - 
\frac{1}{\kappa}\un{\widetilde{w}}\s
\end{aligned}\right.
\end{equation}
\end{cor}

One has to pay attention to the computation of the parameter $\kappa$ which 
implies exchanges between subdomains. Indeed, this coefficient is defined by:
\begin{equation}
\kappa=\frac{\sqrt{\sum_s\ecr{{\depadj}_{D}\s,\hat{\sigadj}_{N}\s}{\Omega\s}}}{
\sqrt{\sum_s\ecr{\dep_D\s,\hat{\sig}_N\s}{\Omega\s}}}
\end{equation}
Anyhow this exchange between subdomains is already done at the end of the parallel error estimation to obtain global measures.

Despite the possibility to separate contributions in terms
$\sqrt{\sum_s\ecr{{\depadj}_{D}\s,\hat{\sigadj}_{N}\s}{\Omega\s}}$ and $
\sqrt{\sum_s\ecr{\dep_D\s,\hat{\sig}_N\s}{\Omega\s}}$, the full separation in the lower bounds
$\beta^-_{inf}$ and $\beta^-_{inf}$ is a complex task since the parameter
$\kappa$ is the ratio of 
errors mixing algebraic and discretization sources.   

However, the separation of sources for both global errors enables steering the iterative solver by an objective of precision of the quantity of 
interest (see~\cite{Rey15}). The computation of $\beta^-_{inf}$ and 
$\beta^-_{inf}$  
after convergence improves the bounding.

\vspace{-6pt}
\section{Numerical assessment}\label{sec:assessment}
\vspace{-2pt}
For all numerical examples, the behavior is linear, isotropic and elastic. The Young modulus is 1 Pa and the Poisson 
coefficient is 0.3.

\subsection{Structure with exact solution}
Let us consider a square linear elastic structure 
$\Omega=[-3l;3l]\times[-3l;3l]$ with homogeneous Dirichlet boundary conditions and plane strain hypothesis.
The domain is subjected 
to a polynomial body force such that the exact solution is known: 
\begin{equation*}
\underline{u}_{ex} = (x+3l)(x-3l)(y+3l)(y-3l)\left((y-3l)^2 \underline{e}_x + 
(y+3l)\underline{e}_y \right)
\end{equation*}
The mesh is made out of first order Lagrange triangles.
As shown in Figure~\ref{fig:poutre_decoupage} the structure is decomposed into 9 
regular subdomains.
 \begin{figure}[ht]
\centering
\includegraphics[width=.3\textwidth]{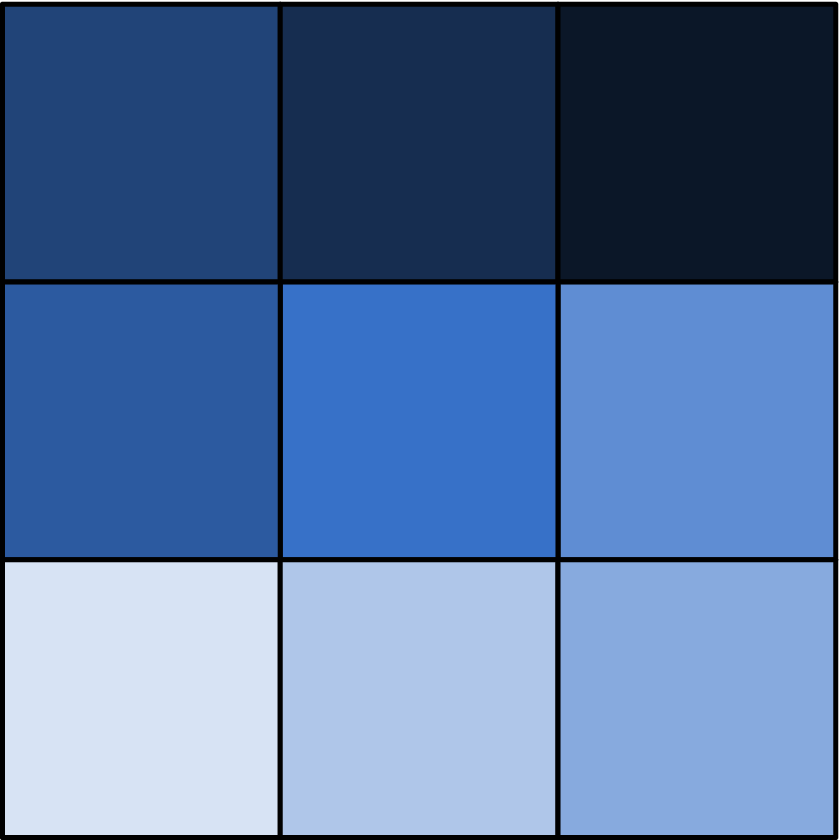}
\caption {
Substructuring}\label{fig:poutre_decoupage}
 \end{figure}

We solve  the problem with a BDD solver (primal approach). We use  the 
Flux-free technique \cite{Par06} to build statically admissible stress fields; 
each star-patch problem is solved by subdividing each element into 12 elements (h-refinement technique).
For the sake of simplicity, we note: 
\begin{itemize}
 \item $\rho=\frac{|L(\un{w})-a(\un{u}_D,\un{w})|}{ 
\sqrt{\sum_s\vvvert  
\un{w}\s \vvvert_{\Omega\s}^2} } $ the lower bound of the error
\item $\rho_{discr}=\frac{|L(\un{w})-a(\un{u}_N,\un{w})|}{ 
\sqrt{\sum_s\vvvert  
\un{w}\s \vvvert_{\Omega\s}^2} }$ the discretization part of the lower bound
\item $\rho_{alg}=\frac{|a(\un{u}_D-\un{u}_N,\un{w})|}{ \sqrt{\sum_s\vvvert  
\un{w}\s \vvvert_{\Omega\s}^2}} $ the algebraic part of the lower bound

\item $\rho_{bis}=\rho_{discr} -{\alpha}$ the lower bound with separation of sources of error

\item $\theta=\sqrt{\sum_s\ecrc{\dep_D\s,\hat{\sig}_N\s}{\Omega\s}}$ the upper bound of the error
\item $\theta_{discr}=\sqrt{\sum_s\ecrc{\dep_N\s,\hat{\sig}_N\s}{\Omega\s}}$ the discretization part of the upper bound
\end{itemize}

The quantities with the superscript $^{seq}$ are computed with a sequential 
simulation (no substructuring and use of a direct solver).

\subsubsection{Effects of the substructuring on the computation of the lower bound at convergence}
In this subsection, we compare the lower bound obtained in a sequential 
simulation with the lower bound obtained in the substructured context when the 
solver {has converged}. Since the study is done at convergence, the primal and dual approaches are 
equivalent.

\begin{figure}[ht]
\centering
\begin{tikzpicture}
\begin{loglogaxis}[
scale=1,
xlabel=$h$,
legend style={at={(.63,0.03)}, anchor=south west}]
\addplot[color=black,mark=none, dashed] table[x=h,y=binf_seq] 
{cube_binf_globale_convh.txt};
\addlegendentry{\begin{footnotesize}$\rho^{seq}$\end{footnotesize}}
\addplot[color=black,only marks, mark=square] table[x=h,y=binf_dd] 
{cube_binf_globale_convh.txt};
\addlegendentry{\begin{footnotesize}{$\rho$}\end{footnotesize}}
\addplot[color=black,mark=none] table[x=h,y=bsup_seq] 
{cube_binf_globale_convh.txt};
\addlegendentry{\begin{footnotesize}{$\theta^{seq}$}\end{footnotesize}}
\addplot[color=black,only marks, mark=star] table[x=h,y=bsup_dd] 
{cube_binf_globale_convh.txt};
\addlegendentry{\begin{footnotesize}{$\theta$}\end{footnotesize}}
\addplot[draw=violet, mark=triangle*] table[x=h,y=evraie] 
{cube_binf_globale_convh.txt};
\addlegendentry{\begin{footnotesize}$\vvvert 
\un{u}_{ex}-\depH \vvvert_\Omega$\end{footnotesize}}
\addplot[color=brown] coordinates {
(0.075, 0.5)
(0.075, 0.9375)
(0.04 , 0.5)
(0.075, 0.5)
};
\addlegendentry{\begin{footnotesize}h-slope\end{footnotesize}}
\end{loglogaxis}
\draw (-0.5,2.5) node[scale=1.,rotate=90]{Error};
\end{tikzpicture}\caption{Evolution of the upper and lower bounds of the error 
in function of the mesh size h}\label{fig:cube_binf_globale_convh}
\end{figure}
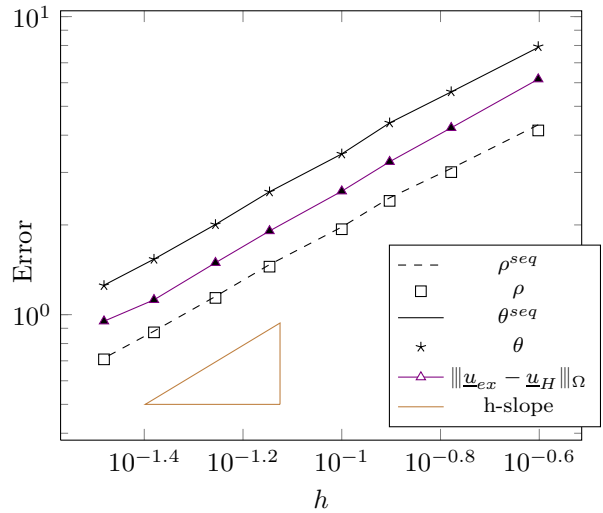

On figure~\ref{fig:cube_binf_globale_convh}, we observe that the bounds are the same for sequential and 
substructured 
computations. We also verify that the exact error is between upper and lower bounds. The 
convergence is the one expected for such a regular problem (h-slope).

Then, we compare the  bounds for several substructuring as 
illustrated in figure~\ref{fig:cube_sous-structurations}. Table~\ref{tab:cube_sous-structurations_resultats_rel} 
gathers the lower bounds 
for sequential, primal and dual approaches computed at convergence normalized by the bounds for sequential 
computation. We 
observe 
that the substructuring has quasi no influence on the accuracy of the lower bound.
\begin{figure}[ht]
\centering
\includegraphics[width=.49\textwidth]{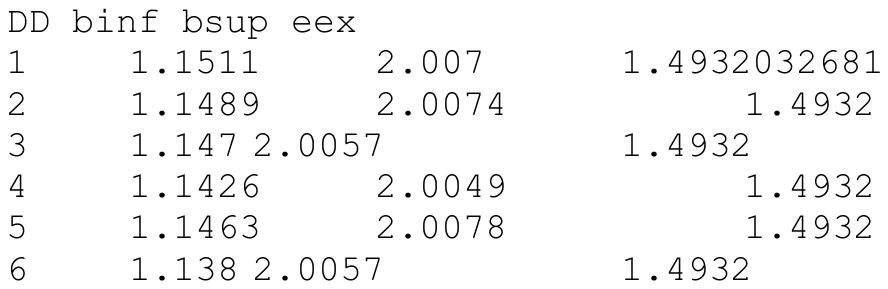}\caption{
Different substructurings} \label{fig:cube_sous-structurations}
\end{figure}



\begin{table}[ht]\centering%
\begin{tabular}{|c|c|c|c|c|c|c|}
\hline 
 & seq & b & c & d & e & f  \\
\hline 
$\frac{\theta}{\theta^{seq}}$ & 1 & 1.0002 &0.9993 &   0.9989 &  1.0004 & 0.9993\\
\hline 
$\frac{\rho}{\rho^{seq}}$ & 1& 0.9981 & 0.9964& 0.9926& 0.9958& 0.9886 \\
\hline
\end{tabular}\caption{Relative upper and lower bounds of the error for 
various substructuring}\label{tab:cube_sous-structurations_resultats_rel}
\end{table}

\subsubsection{Separation of sources of error}
In this subsection, we illustrate the separation of sources of error in 
the lower bound.

On the first graph in figure~\ref{fig:cube_binf_globale_separation}, we give the evolution of the upper and lower 
bounds and of the true error until the fifth iteration. We observe the fast convergence of those bounds. 
On the same graph, we also visualize the discretization parts $\theta_{discr}$ and
$\rho_{discr}$. The bounding is more precise and enables to define after a 
few iterations the interval in which the true error is located at convergence. The second graph in 
figure~\ref{fig:cube_binf_globale_separation} gives the 
evolution of the terms in theorem~\ref{thm:borne_inf_dd_separation_1}. The quantity $\alpha$ 
scrictly decreases along the iterations, the evolution 
of the quantity $\rho_{alg}$ is not as smooth.

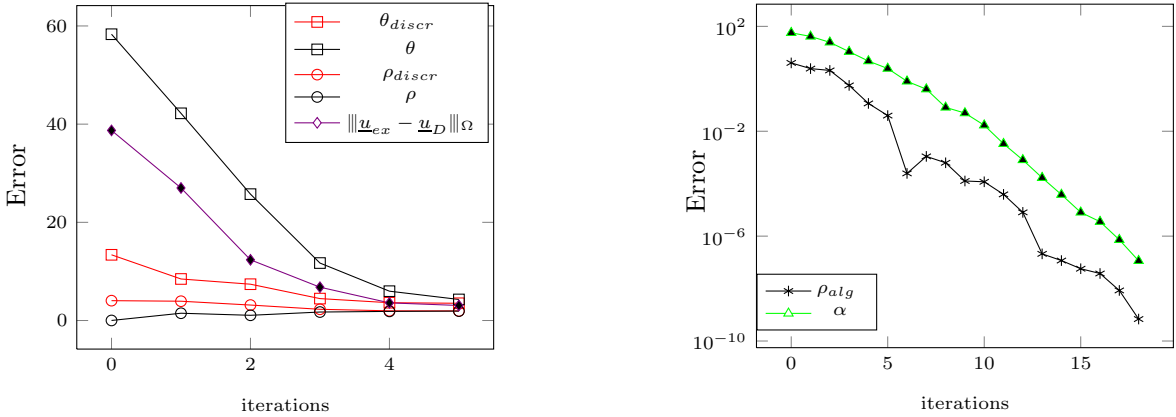
\begin{figure}[ht]
\begin{minipage}{.49\textwidth}
\begin{tikzpicture}
\begin{axis}[
every axis/.append style={font=\scriptsize},
scale=0.8,
xlabel=iterations,
legend style={at={(0.5,0.6)}, anchor=south west}]
\addplot [draw=red, mark=square] table[x=iterations,y=bsup_discr] 
{cube_binf_globale_iter_coupe.txt};
\addlegendentry{ {\scriptsize $\theta_{discr}$ } }
\addplot [draw=black, mark=square] table[x=iterations,y=bsup] 
{cube_binf_globale_iter_coupe.txt};
\addlegendentry{{\scriptsize  $\theta$}}
\addplot [draw=red, mark=o] table[x=iterations,y=binf_discr] 
{cube_binf_globale_iter_coupe.txt};
\addlegendentry{{\scriptsize $\rho_{discr}$}}
\addplot [draw=black, mark=o] table[x=iterations,y=binf] 
{cube_binf_globale_iter_coupe.txt};
\addlegendentry{{\scriptsize  $\rho$}}
\addplot [draw=violet, mark=diamond*] table[x=iterations,y=evraie] 
{cube_binf_globale_iter_coupe.txt};
\addlegendentry{{\scriptsize  $\vvvert  \un{u}_{ex}-\un{u}_D  \vvvert 
_{\Omega}$}}
\end{axis}
\draw (-0.8,2.3) node[scale=1.,rotate=90]{Error};
\end{tikzpicture}
\end{minipage}
\begin{minipage}{.49\textwidth}
\begin{tikzpicture}
\begin{semilogyaxis}[
every axis/.append style={font=\scriptsize},
scale=0.8,
xlabel=iterations,
legend style={at={(0,0.05)}, anchor=south west}]
\addplot [draw=black, mark=asterisk]  table[x=iterations,y=binf_algebrique] 
{cube_binf_globale_iter.txt};
\addlegendentry{ {\scriptsize $\rho_{alg}$ } }
\addplot [draw=green, mark=triangle*]
table[x=iterations,y=res] 
{cube_binf_globale_iter.txt};
\addlegendentry{{\scriptsize  $\alpha$}}
\end{semilogyaxis}
\draw (-0.8,2.2) node[scale=1.,rotate=90]{Error};
\end{tikzpicture}
\end{minipage}
\caption{Separation of sources of error 
 in upper and lower bounds}\label{fig:cube_binf_globale_separation}
\end{figure}

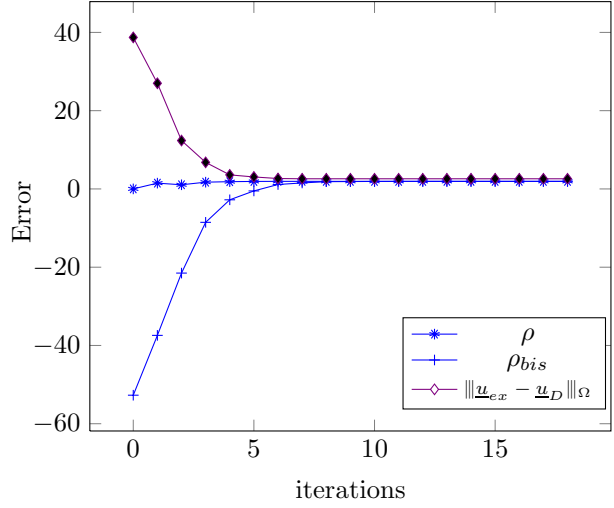
\begin{figure}[ht]
\centering
 \begin{tikzpicture}
\begin{axis}[
scale=1,
xlabel=iterations,
legend style={at={(0.6,0.05)}, anchor=south west}]
\addplot [draw=blue, mark=10-pointed star] table[x=iterations,y=binf] 
{cube_binf_globale_iter.txt};
\addlegendentry[ anchor=mid]{  
$\rho $ }
\addplot [draw=blue, mark=+] table[x=iterations,y=difference2] 
{cube_binf_globale_iter.txt};
\addlegendentry[ anchor=mid]{ $\rho_{bis}$ }
\addplot [draw=violet, mark=diamond*] table[x=iterations,y=evraie] 
{cube_binf_globale_iter.txt};
\addlegendentry{{\scriptsize  $\vvvert  \un{u}_{ex}-\un{u}_D  \vvvert 
_{\Omega}$}}
\end{axis}
\draw (-0.9,2.9) node[scale=1.,rotate=90]{Error};
\end{tikzpicture}
\caption{Evoluation of the two lower bounds during iterations}\label{fig:cube_binf_deux_separations}
\end{figure}

Figure~\ref{fig:cube_binf_deux_separations} represents the two lower bounds 
of theorem~\ref{thm:borne_inf_dd_separation_1}. As 
expected, the second lower bound is not precise at the beginning since it 
gives a negative value. However, the zero of this bound enables to tell 
the moment when the algebraic error becomes smaller than the discretization error.

\subsection{Pre-cracked structure}
We now consider a pre-cracked structure decomposed into 16 subdomains as 
illustrated in Figure~\ref{fig:Diez_adjoint}. The displacements at the base of 
the structure and on the larger hole are imposed to be zero. The upper-left 
part 
and the second hole are subjected to a constant unit pressure. We made the hypothesis of plane stress. The quantity 
of 
interest is the mean of the stress component $\sigma_{xx}$  on a region 
$\omega$ 
close to the crack. In Figure~\ref{fig:Diez_adjoint}, the loading of the 
reference problem is in blue and the loading of adjoint problem is in orange. 
We 
used the FETI algorithm (dual approach) to solve the interface problem and the 
statically admissible stress fields are built using the Flux-free technique 
\cite{Cot09,Par06} with h-refinement technique for the {solving} of local 
problems on star-patches (each element is divided into 16 elements).

\begin{figure}[ht]
\centering
\includegraphics[width=.5\textwidth]{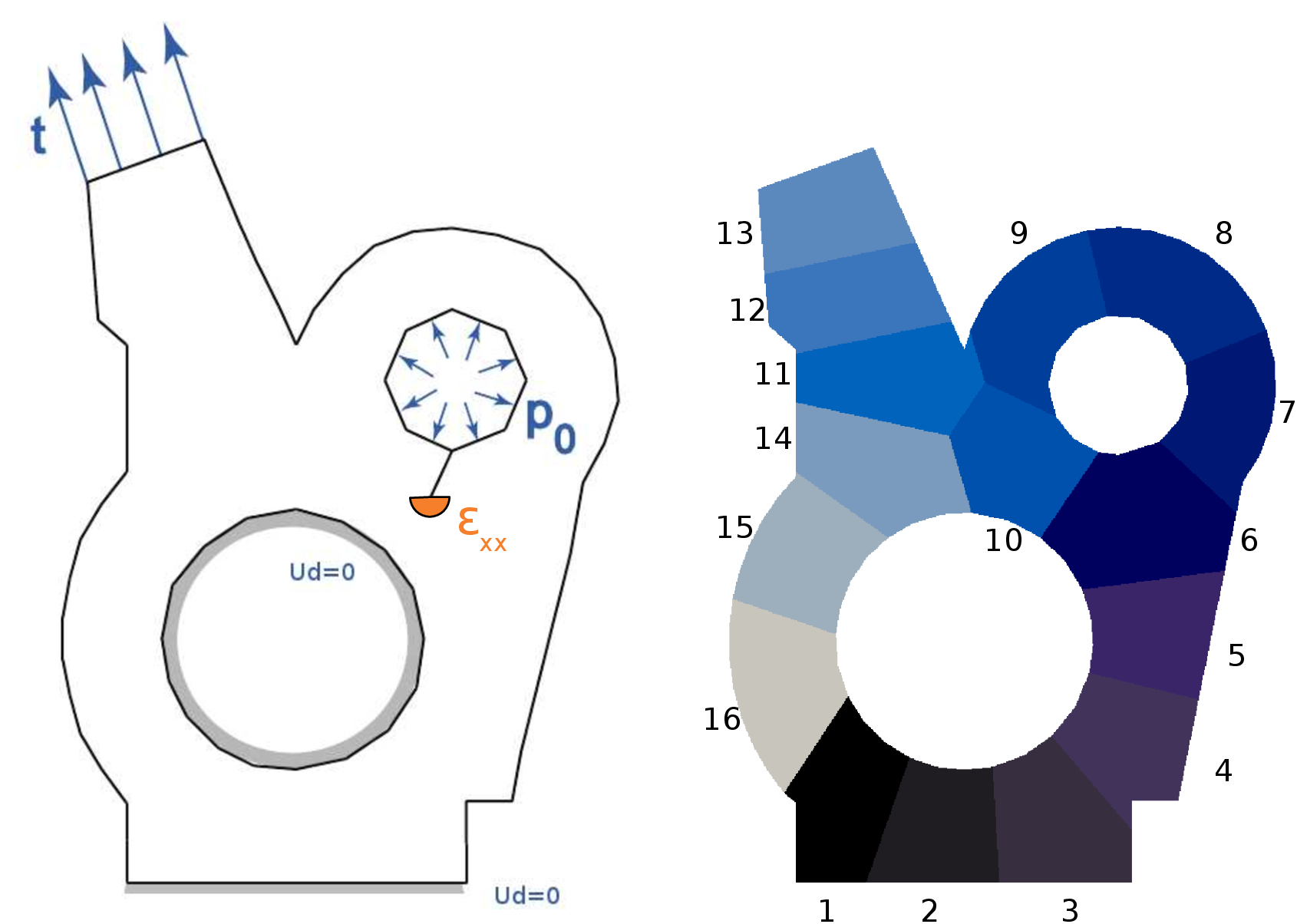}\caption{{
\footnotesize Loading of forward (blue) and adjoint problems (orange), domain 
decomposition}}\label{fig:Diez_adjoint}
\end{figure}

\subsubsection{Separation of sources of error in upper and lower bounds}

For this paragraph, we consider only the forward problem. The mesh used for this computation is composed of 4370 
degrees of freedom. We compute the global upper and lower bounds during the 
iterations and the discretization upper and lower bounds along the iterations on the same graph in 
figure~\ref{fig:diez_binf_globale_separation}.

\begin{figure}[ht]
\begin{minipage}{.49\textwidth}
\begin{tikzpicture}
\begin{axis}[
every axis/.append style={font=\scriptsize},
scale=0.8,
xlabel=iterations,
legend style={at={(0.5,0.6)}, anchor=south west},]
\addplot [draw=red, mark=square] table[x=iterations,y=bsup_discr] 
{diez_binf_globale_iter.txt};
\addlegendentry{ {\scriptsize $\theta_{discr}$ } }
\addplot [draw=black, mark=square] table[x=iterations,y=bsup] 
{diez_binf_globale_iter.txt};
\addlegendentry{{\scriptsize  $\theta$}}
\addplot [draw=red, mark=o] table[x=iterations,y=binf_discr] 
{diez_binf_globale_iter.txt};
\addlegendentry{{\scriptsize $\rho_{discr}$}}
\addplot [draw=black, mark=o] table[x=iterations,y=binf] 
{diez_binf_globale_iter.txt};
\addlegendentry{{\scriptsize  $\rho$}}
\end{axis}
\draw (-0.9,2.0) node[scale=1.,rotate=90]{{\small Error}};
\end{tikzpicture}
\end{minipage}
\begin{minipage}{.49\textwidth}
\begin{tikzpicture}
\begin{semilogyaxis}[
every axis/.append style={font=\scriptsize},
scale=0.8,
xlabel=iterations,
legend style={at={(0,0.05)}, anchor=south west}]
\addplot[draw=black, mark=asterisk] table[x=iterations,y=binf_algebrique] 
{diez_binf_globale_iter.txt};
\addlegendentry{ {\scriptsize $\rho_{alg}$ } }
\addplot [draw=green, mark=triangle*] 
table[x=iterations,y=res] 
{diez_binf_globale_iter.txt};
\addlegendentry{{\scriptsize  $\alpha$}}
\end{semilogyaxis}
\draw (-0.9,2.0) node[scale=1.,rotate=90]{{\small Error}};
\end{tikzpicture}
\end{minipage}
\caption{Pre-cracked structure : Evolution of the upper and lower bounds 
(global, discretization part, algebraic part) and of the residual during the 
iterations}\label{fig:diez_binf_globale_separation}
\end{figure}
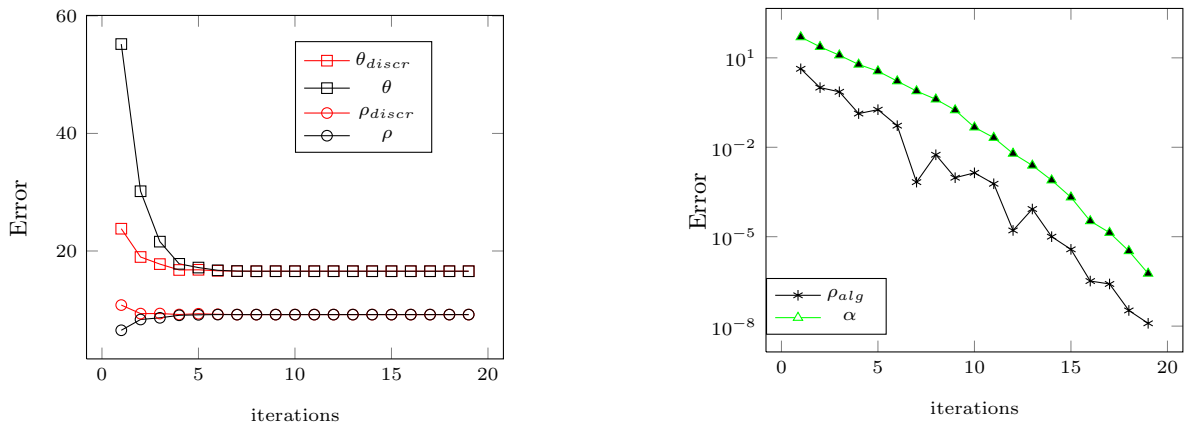

As for the previous example, we observe the fast convergence of the upper and lower bounds. This graph also illustrates 
the separation of sources of error. Finally, we observe that the residual $\alpha$ decreases with the iterations and 
that, again, the evolution of $\rho_{alg}$ is not as regular.

In figure~\ref{fig:diez_binf_globale_separation_zoom}, we give the discretization upper and lower bounds, which define a discretization envelope, for all iterations on the first graph and until the seventh iteration on the second graph.

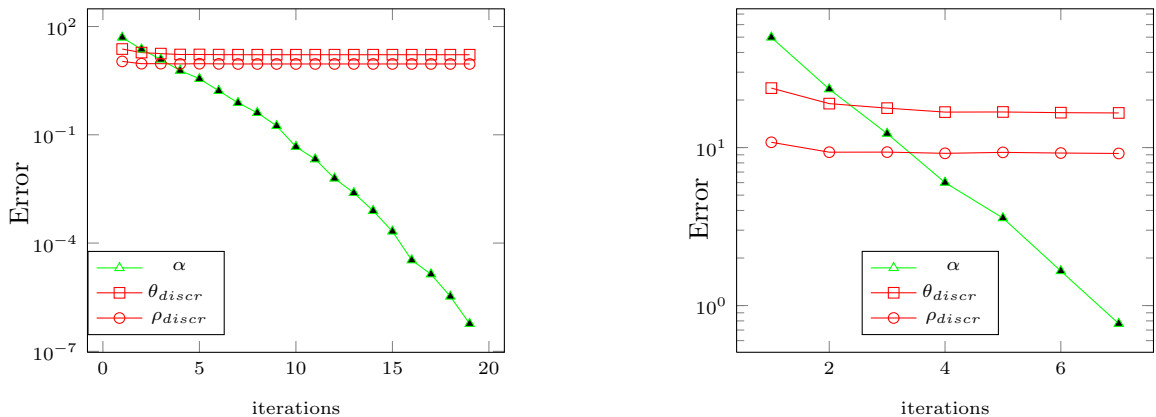
\begin{figure}[ht]
\begin{minipage}{.49\textwidth}
\begin{tikzpicture}
\begin{semilogyaxis}[
every axis/.append style={font=\scriptsize},
scale=0.8,
xlabel=iterations,
legend style={at={(0,0.05)}, anchor=south west}]
\addplot [draw=green, mark=triangle*] 
table[x=iterations,y=res] 
{diez_binf_globale_iter.txt};
\addlegendentry{{\scriptsize  $\alpha$}}
\addplot [draw=red, mark=square] table[x=iterations,y=bsup_discr] 
{diez_binf_globale_iter.txt};
\addlegendentry{ {\scriptsize $\theta_{discr}$ } }
\addplot [draw=red, mark=o] table[x=iterations,y=binf_discr] 
{diez_binf_globale_iter.txt};
\addlegendentry{{\scriptsize $\rho_{discr}$}}
\end{semilogyaxis}
\draw (-0.9,2.1) node[scale=1.,rotate=90]{Error};
\end{tikzpicture}
\end{minipage}
\begin{minipage}{.49\textwidth}
\begin{tikzpicture}
\begin{semilogyaxis}[
every axis/.append style={font=\scriptsize},
scale=0.8,
xlabel=iterations,
legend style={at={(0.3,0.05)}, anchor=south west}]
\addplot [draw=green, mark=triangle*] 
table[x=iterations,y=res] 
{diez_binf_globale_iter_coupe.txt};
\addlegendentry{{\scriptsize  $\alpha$}}
\addplot [draw=red, mark=square] table[x=iterations,y=bsup_discr] 
{diez_binf_globale_iter_coupe.txt};
\addlegendentry{ {\scriptsize $\theta_{discr}$ } }
\addplot [draw=red, mark=o] table[x=iterations,y=binf_discr] 
{diez_binf_globale_iter_coupe.txt};
\addlegendentry{{\scriptsize $\rho_{discr}$}}
\end{semilogyaxis}
\draw (-0.5,1.9) node[scale=1.,rotate=90]{Error};
\end{tikzpicture}
\end{minipage}
\caption{Pre-cracked structure : Discretization envelope and residual against 
iterations}\label{fig:diez_binf_globale_separation_zoom}
\end{figure}

The discretization envelope can be used to define the stopping criterion of 
the solver. In \cite{Rey14}, the proposed criterion was  to stop when the 
algebraic error 
was ten times smaller than the discretization error. On this example, the solver would stop at the sixth iteration. 
A new stopping criterion 
could be to stop when the algebraic error is smaller than the 
discretization part of the lower bound, which would lead to stop to at the fourth iteration.

\subsection{Goal-oriented error estimation}

In this subsection, we propose an auto-adaptive strategy to steer the iterative solver by an objective of 
precision on a quantity of interest. In this example, the objective of precision will be five percent. 
We used the FETI algorithm (dual approach) to solve the interface problem and the 
statically admissible stress fields are built using the Flux-free technique 
\cite{Par06} with h-refinement technique for the {solving} of local 
problems on star-patches (each element is divided into 4 elements). Since the forward and adjoint problems are 
auto-adjoint (due to the symetry of the bilinear form), we solve the two problems simultaneously using a block 
conjugate gradient (see \cite{Rey15} for more details). The separation of sources of error in the lower bounds 
of global error on forward and adjoint problem enables the definition of the following criterion : \textbf{STOP when 
$\alpha < \rho_{discr}$ and $\widetilde{\alpha} < \widetilde{\rho}_{discr}$} which expresses the fact that the 
residual is out of the discretization enveloppe so that the algebraic error is 
negligible in comparaison with the discretization error. Using equation \eqref{eq:id_parallelogramme_appli}, we have 
the follwing upper and lower bounds on the unknown exact value of the quantity of interest $I_{ex}$: 
\begin{equation}\label{eq:id_parallelogramme_appli_2}
I_{ex}^-=I_H+\frac{1}{4} \beta^+_{inf} -\frac{1}{4}  \beta^-_{sup}   \leq 
I_{ex}  \leq I_H+\frac{1}{4} \beta^+_{sup} -\frac{1}{4}  
\beta^-_{inf} =I_{ex}^+
\end{equation}
We will compare the bounds $I_{ex}^-$ and $I_{ex}^+$ in case the quantities $\beta^+_{inf}$ and $\beta^-_{inf}$ are 
computed using expressions in Corollary~\ref{prop:S+} and in case they are chosen equal to zero 
($\beta^+_{inf}=\beta^-_{inf}=0$), which is equivalent to the bounding in equation \eqref{eq:majoration_IHH2} with 
$\dep_H=\dep_D$, $\depadj_{\widetilde{H}}=\depadj_D$, $\hat{\sig}_H=\hat{\sig}_N$ and 
$\hat{\sigadj}_{\widetilde{H}}=\hat{\sigadj}_N$.

We start with a first mesh which is a little bit refined near the quantity of interest in order to have several 
elements in the region $\omega$. The discretization error is computed at iteration 1 
(which is more relevant than the initilization Iteration 0). The criterion is defined and the solver iterates until 
the criterion is reached. The error is estimated once again to verify that the discretization error has not changed 
too much. We give in table~\ref{tab:diez_res1} the evolution of the residual for forward and 
adjoint problems and the bounds on the global errors for the two problems.

\begin{table}[ht]\centering%
\begin{tabular}{|c|c|c|c|c|c|c|}

\hline 
 Iteration &$\theta_{discr}$  & $\rho_{discr}$ & $\widetilde{\theta}_{discr}$ & $\widetilde{\rho}_{discr}$  & $\alpha  
$&$\widetilde{\alpha}$  \\ 
\hline 
0 & & & & & 260.07 & 0.26041\\
\hline 
1 & 11.45 & 7.0008& 0.12867& 1.5457 $10^{-3}$& 48.339 & 0.14931\\
\hline 
2 & & & & & 35.991 & 8.5664 $10^{-2}$\\
\hline 
3 & & & & &16.649& 4.8205 $10^{-2}$\\
\hline 
4 & & & & & 5.5966 &1.2462 $10^{-2}$\\
\hline 
5 & & & & & 3.5765 & 8.4996 $10^{-3}$\\
\hline 
6 & 9.9004 & 6.7105 & 0.12682 & 1.5893 $10^{-3}$ & 0.94805& 1.5156 $10^{-3}$\\
\hline 
\end{tabular} 
\caption{Pre-cracked structure: First {mesh}}\label{tab:diez_res1}
\end{table}

Regarding the quantity of interest, at the sixth iteration, we obtain the values presented in 
table~\ref{tab:diez_res1_donnes}.

\begin{table}[ht]\centering%
\begin{tabular}{|c|c|c|c|c|}

\hline 
 $I_H$ &$\beta^-_{inf} $ & $\beta^+_{inf}$ & $\frac{1}{4}\beta^+_{sup}$ & $\frac{1}{4}\beta^-_{sup}$   \\ 
\hline 
3.1505 & 1.0087&1.2129 & 0.68424&0.57132 \\

\hline 
\end{tabular} 
\caption{Pre-cracked structure: First {mesh} : error estimation}\label{tab:diez_res1_donnes}
\end{table}

In table~\ref{tab:diez_res1_parallelogramme}, we give the upper and lower bounds of the unknown exact value $I_{ex}$ 
of the quantity of interest with and without the use of the lower bounds. We observe that the use of the lower bounds 
enables to reduce the width by 44 \%. At the end of the first {solving}, the 
error on the quantity of interest is 22.224 \%. 
\begin{table}[ht]\centering%
\begin{tabular}{|c||c|c|c|c|}
\hline 
    & {\footnotesize $I_{ex}^+$}&  {\footnotesize  $I_{ex}^-$} & {\footnotesize width }& 
 {\footnotesize precision  } \\ 
\hline 
 {\footnotesize With $\beta_{inf}^+=\beta_{inf}^-=0$} & 3.8347 &2.5791& 1.2555 & 39.852 \% \\
\hline 
{ \footnotesize With $\beta_{inf}^+$ and $\beta_{inf}^-$ from table~\ref{tab:diez_res1_donnes}}   &  3.5825 &2.8824 & 
0.7001 & 22.224 \%  \\
\hline 
\end{tabular} 
\caption{Pre-cracked structure: First {mesh} : bounds for the exact 
quantity of interest}\label{tab:diez_res1_parallelogramme}
\end{table}

We also give, for both problems, the contribution from each subdomain thanks to the following quantities : 
\begin{equation}
 \eta\s=\frac{\ecr{\dep_N\s,\hat{\sig}_N\s}{\Omega\s}}{\sqrt{\sum_s\ecrc{\dep_N\s,\hat{\sig}_N\s}{\Omega\s}}} \qquad 
\text{and}\qquad
\widetilde{\eta}\s=\frac{\ecr{\depadj_N\s,\hat{\sigadj}_N\s}{\Omega\s}}{\sqrt{\sum_s\ecrc{\depadj_N\s,\hat{\sigadj}_N\s
}{\Omega\s}}}
\end{equation}

and plot the contributions on the Figure~\ref{Fig:Diez_resolution1_hist}.
\begin{figure}\centering
\begin{tikzpicture}[scale=.8]
\begin{axis} [width= 0.8\textwidth, ybar,bar width=5pt,xtick={1,2,3,4,5,6,7, 8, 9, 10, 11, 12, 13, 14, 15, 16},legend 
style={at={(0,1)},anchor=north west}]
  \addplot [fill=blue,draw=blue] table[x=sd,y=e_ref_rel]{diez_res1.txt};
  \addplot table[x=sd,y=e_adj_rel]{diez_res1.txt};
\legend{$\eta\s$, $\widetilde{\eta}\s$}
\node at (axis cs:2,62.8) {{\small $\theta_{discr}=9.9004$}};
\node at (axis cs:2.1,52.7) {{\small $\widetilde{\theta}_{discr}=0.12682$}};
\end{axis}
\end{tikzpicture}
\caption{First {mesh} : Distribution of the error estimator within 
subdomains}\label{Fig:Diez_resolution1_hist}
\end{figure}
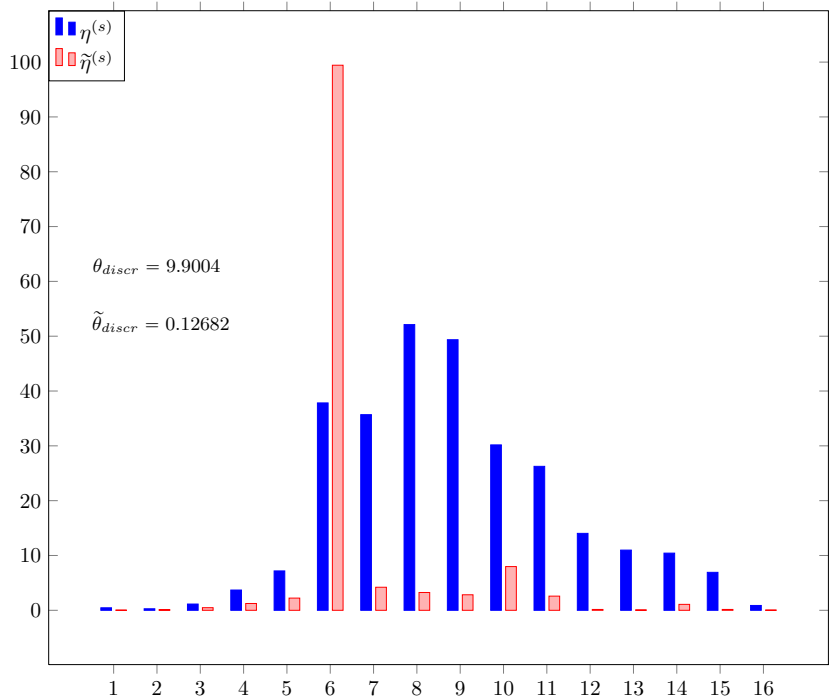

As expected, for the adjoint problem, the error is mainly located in the sixth subdomain, which is the subdomain with 
the load. For the forward problem, the error is more diffuse. 

In order to reach the objective of precision, we decide to first improve the quality of the solution of the forward 
problem. To do so, since the error is not located in few subdomains, we decide to refine the mesh on the whole 
structure. Of course, we could have used the error map provided by the error estimator and a remeshing criterion to 
process adaptive remeshing (see for instance \cite{Ver96,Die99,Ode01, Bel05, Die07}). For sake of simplicity, a 
refinement by splitting is performed. 
We also reuse the search directions computed during the first {solving} in order to speed up the next 
{one}. The 
12 interface vectors (2 interface vectors -one for the forward problem, one for the adjoint problem- 
computed at each iteration of the first {solving} that converged in 6 iterations) corresponding to the 
search 
directions are projected on the new mesh and used as additional constraints thanks to augmented-Krylov methods 
\cite{Saad97}. Since the dual approach was chosen, we use a two-level FETI algorithm \cite{Far98b} to take into account 
the additional constraints. It does not modify the methodology to construct admissible fields nor the error estimator. 
Recycling search directions leading to a better initilization, we compute the error estimation at iteration 0 
and define the stopping criterion. Once the criterion is reached, the error is estimated once again and if the 
criterion is checked, the solver is stopped. We give in the table~\ref{tab:diez_res2} the 
evolution of the residuals and of the bounds on the global errors for the forward and adjoint problems. We can observe 
that the first residual is comparable to the last residual of the first {solving}.

\begin{table}[ht]\centering%
\begin{tabular}{|c|c|c|c|c|c|c|}

\hline 
 Iteration &$\theta_{discr} $ & $\rho_{discr}$ &$ \widetilde{\theta}_{discr}$ & $\widetilde{\rho}_{discr}$  &$ \alpha  
$&$\widetilde{\alpha}$  \\ 
\hline 
0 & 5.7611& 3.8882& 7.5633 $10^{-2}$&7.551 $10^{-4}$& 5.7017 & 7.4954 $10^{-3}$\\
\hline 
1 &   &  & &   & 3.3265 & 4.4442 $10^{-3}$\\
\hline 
2 & & & & & 1.2869 & 8.4996 $10^{-4}$\\
\hline 
3 & 5.4304& 3.7882& 7.5564 $10^{-2}$ &7.6332 $10^{-4}$ &0.31197& 2.4632 $10^{-4}$\\
\hline 

\end{tabular} 
\caption{Pre-cracked structure: Second {mesh}}\label{tab:diez_res2}
\end{table}

Regarding the quantity of interest, we obtain the values presented in table~\ref{tab:diez_res2_donnes}.

\begin{table}[ht]\centering%
\begin{tabular}{|c|c|c|c|c|}

\hline 
 $I_H$ &$\beta^-_{inf} $ & $\beta^+_{inf}$ & $\frac{1}{4}\beta^+_{sup}$ & $\frac{1}{4}\beta^-_{sup}$   \\ 
\hline 
3.2266 & 0.32673&0.43838& 0.23661&0.17372  \\

\hline 
\end{tabular} 
\caption{Pre-cracked structure: Second {mesh} : error estimation}\label{tab:diez_res2_donnes}
\end{table}

In table~\ref{tab:diez_res2_parallelogramme}, we give the upper and lower bounds of the unknown exact value $I_{ex}$ 
of the quantity of interest with and without the use of the lower bounds. We observe that the use of the lower bounds 
improves the bounding. At the end of the second {solving}, the uncertainty on the quantity of interest 
is 6.7893 \%. 

\begin{table}[ht]\centering%
\begin{tabular}{|c||c|c|c|c|c|}
\hline 
    & {\footnotesize  $I_{ex}^+$}&  {\footnotesize   $I_{ex}^-$} & {\footnotesize width }& 
 {\footnotesize precision  } \\ 
 \hline 
 {\footnotesize With $\beta_{inf}^+=\beta_{inf}^-=0$} & 3.4632 &3.0528& 0.41034 & 12.717 \% \\
\hline 
{ \footnotesize With $\beta_{inf}^+$ and $\beta_{inf}^-$ from table~\ref{tab:diez_res2_donnes}}   &  3.3815 &3.1624 & 
0.21906 & 6.7893 \%  \\
\hline 
\end{tabular} 
\caption{Pre-cracked structure: Second {mesh} : bounds for the exact 
quantity of interest}\label{tab:diez_res2_parallelogramme}
\end{table}


In order to reach the objective of precision, we decide to refine the mesh only for the sixth subdomain to improve the 
quality of the adjoint solution. Refining this subdomain's discretization introduces incompatibilities at the 
interface that can be easily managed thanks to transfer matrix as explained in \cite{Rey15}. This incompatibility does 
not affect the error estimator since the quantity of interest is located far from the subdomain's boundary. Once 
again, we reuse the search directions of the first two {solutions} to speed up the third 
{solving}. As the 
discretization of 
the interface is not modified (see \cite{Rey15}), the interface vectors can be directly used as additional 
constraints.

\begin{table}[ht]\centering%
\begin{tabular}{|c|c|c|c|c|c|c|}

\hline 
 Iteration &$\theta_{discr}$  & $\rho_{discr}$ & $\widetilde{\theta}_{discr}$ & $\widetilde{\rho}_{discr}$  & $\alpha  
$&$\widetilde{\alpha}$  \\ 
\hline 
0 & 5.1501  & 3.5841& 4.6197  $10^{-2}$ & 2.9485 $10^{-4}$& 0.32764 &  6.2878 $10^{-4}$\\
\hline 
1 & 5.1267 & 3.5824 &4.619  $10^{-2}$&2.9482 $10^{-4}$&  0.10619& 1.6588 $10^{-4}$\\ \hline
\end{tabular} 
\caption{Pre-cracked structure: Third {mesh}}\label{tab:diez_res3}
\end{table}

Regarding the quantity of interest, we obtain the values presented in table~\ref{tab:diez_res3_donnes}. 

\begin{table}[ht]\centering%
\begin{tabular}{|c|c|c|c|c|}

\hline 
 $I_H$ &$\beta^-_{inf} $ & $\beta^+_{inf}$ & $\frac{1}{4}\beta^+_{sup}$ & $\frac{1}{4}\beta^-_{sup}$   \\ 
\hline 
3.2625 & 0.18986&0.25873& 0.13618&0.10058   \\

\hline 
\end{tabular} 
\caption{Pre-cracked structure: Third {mesh} : error estimation}\label{tab:diez_res3_donnes}
\end{table}

In table~\ref{tab:diez_res3_parallelogramme}, we give the upper and lower bounds of the unknown exact value $I_{ex}$ 
of the quantity of interest with and without the use of the lower bounds. We observe that the use of the lower bounds 
enables to reduce the width by 44 \%. At the end of the third {solving}, the objective of precision is 
reached and 
the error on the exact value of the quantity of interest is smaller than 4 \%.
\begin{table}[ht]\centering%
\begin{tabular}{|c||c|c|c|c|}
\hline 
    & {\footnotesize $I_{ex}^+$}&  {\footnotesize $I_{ex}^-$} & {\footnotesize width }& 
 {\footnotesize precision  } \\ 
\hline 
 {\footnotesize With $\beta_{inf}^+=\beta_{inf}^-=0$} & 3.3986 &3.1619& 0.2367 & 7.2574 \% \\
\hline 
{ \footnotesize With $\beta_{inf}^+$ and $\beta_{inf}^-$ from table \ref{tab:diez_res3_donnes}}   &  3.3512 &3.2265 & 
0.1247 & 3.8199\%  \\
\hline 
\end{tabular} 
\caption{Pre-cracked structure: Third {mesh} : bounds for the exact 
quantity of interest}\label{tab:diez_res3_parallelogramme}
\end{table}

Finally, we give the global errors and residuals against cumulative iteration on Figure~\ref{fig:Diez_global} and the 
evolution of the approximated value of the quantity of interest $I_H$ and the upper 
and 
lower bounds for $I_{ex}$ in Figure~\ref{fig:IH_evolution}.

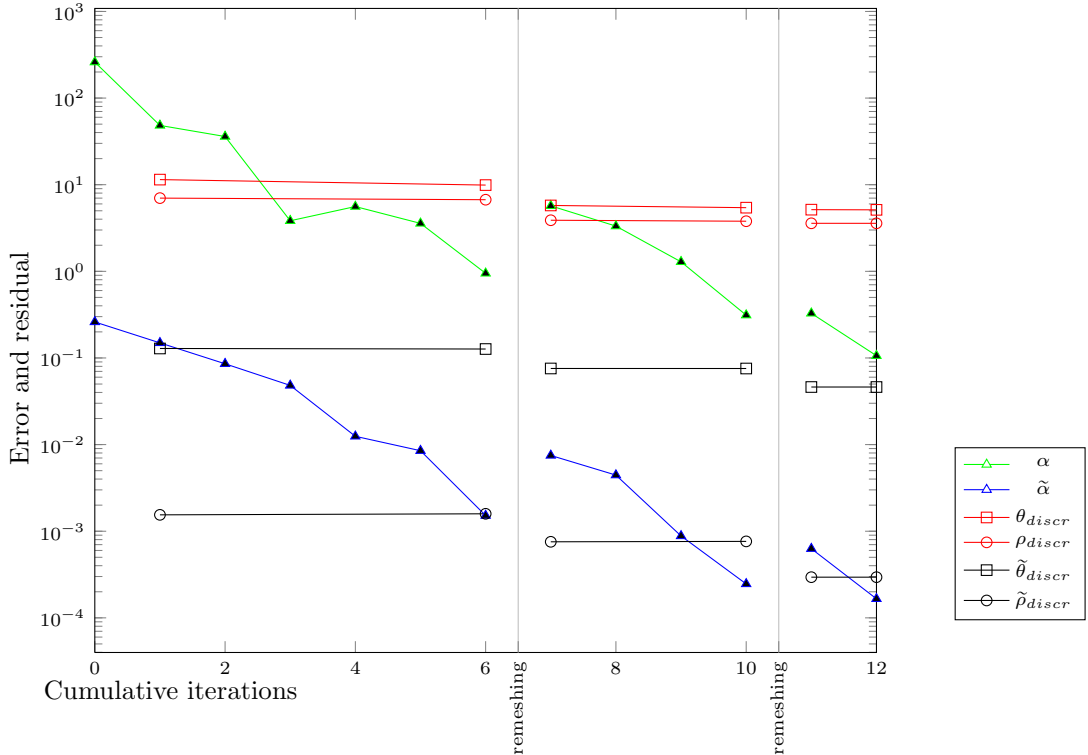
\begin{figure}[ht]\centering
\begin{tikzpicture}
\begin{semilogyaxis}[
every axis/.append style={font=\scriptsize},
scale=1.5,legend style={at={(1.1,0.05)}, anchor=south west},
enlarge x limits=false, extra x ticks={6.5, 10.5},
extra x tick labels={remeshing, remeshing},
extra x tick style={grid=major,
tick label style={rotate=90,anchor=east}},
legend entries={{\scriptsize  $\alpha$}, {\scriptsize  $\widetilde{\alpha}$}, {\scriptsize $\theta_{discr}$ }, 
{\scriptsize $\rho_{discr}$ }, {\scriptsize $\widetilde{\theta}_{discr}$}, {\scriptsize $\widetilde{\rho}_{discr}$}}
]
\addplot [draw=green, mark=triangle*,forget plot] 
table[x=iter_cumul,y=res] 
{diez_resol1.txt };
\addplot [draw=green, mark=triangle*,forget plot] 
table[x=iter_cumul,y=res] 
{diez_resol2.txt};
\addplot [draw=green, mark=triangle*] 
table[x=iter_cumul,y=res] 
{diez_resol3.txt};

\addplot [draw=blue, mark=triangle*,forget plot] 
table[x=iter_cumul,y=res_adj] 
{diez_resol1.txt};
\addplot [draw=blue, mark=triangle*,forget plot] 
table[x=iter_cumul,y=res_adj] 
{diez_resol2.txt};
\addplot [draw=blue, mark=triangle*] 
table[x=iter_cumul,y=res_adj] 
{diez_resol3.txt};

\addplot [draw=red, mark=square] table[x=iter_cumul,y=theta,forget plot] 
{diez_resol1.txt};
\addplot [draw=red, mark=square] table[x=iter_cumul,y=theta,forget plot] 
{diez_resol2.txt};
\addplot [draw=red, mark=square] table[x=iter_cumul,y=theta] 
{diez_resol3.txt};

\addplot [draw=red, mark=o] table[x=iter_cumul,y=rho,forget plot] 
{diez_resol1.txt};

\addplot [draw=red, mark=o] table[x=iter_cumul,y=rho,forget plot] 
{diez_resol2.txt};

\addplot [draw=red, mark=o] table[x=iter_cumul,y=rho] 
{diez_resol3.txt};

\addplot [draw=black, mark=square] table[x=iter_cumul,y=theta_adj,forget plot] 
{diez_resol1.txt};
\addplot [draw=black, mark=square] table[x=iter_cumul,y=theta_adj,forget plot] 
{diez_resol2.txt};
\addplot [draw=black, mark=square] table[x=iter_cumul,y=theta_adj] 
{diez_resol3.txt};

\addplot [draw=black, mark=o] table[x=iter_cumul,y=rho_adj,forget plot] 
{diez_resol1.txt};
\addplot [draw=black, mark=o] table[x=iter_cumul,y=rho_adj,forget plot] 
{diez_resol2.txt};
\addplot [draw=black, mark=o] table[x=iter_cumul,y=rho_adj] 
{diez_resol3.txt};

\end{semilogyaxis}
\draw (-0.99,3.9) node[scale=1.,rotate=90]{Error and residual };
\draw (1,-0.5) node[scale=1.]{Cumulative iterations};
\end{tikzpicture}
\caption{Pre-cracked structure : Discretization envelope and residual against 
cumulative iterations}\label{fig:Diez_global}
\end{figure}

\begin{figure}[ht]\centering
\begin{tikzpicture}
\begin{axis}[
every axis/.append style={font=\scriptsize},
scale=1,
xlabel={meshes},
xtick={1, 2, 3},
legend style={at={(1.1,0.05)}, anchor=south west}]
\addplot [draw=violet,mark=pentagon*, only marks] table[x=resol,y=IH] 
{diez_global_IH.txt};
\addlegendentry{ {\scriptsize $I_H$ } }
\addplot [draw=violet,mark=pentagon, dashed] table[x=resol,y=Iex_bsup] 
{diez_global_IH.txt};
\addlegendentry{{\scriptsize Upper bound of $I_{ex}$}}
\addplot [draw=violet, mark=pentagon, dashed] table[x=resol,y=Iex_binf] 
{diez_global_IH.txt};
\addlegendentry{{\scriptsize Lower bound of $I_{ex}$}}



\end{axis} and residual
\draw (-0.9,2.9) node[scale=1.,rotate=90]{{\scriptsize Error on the quantity of interest}};
\end{tikzpicture}
\caption{Pre-cracked structure : Evolution of the approximated quantity of interest $I_H$ and of the 
bounds of the exact value $I_{ex}$}\label{fig:IH_evolution}
\end{figure}
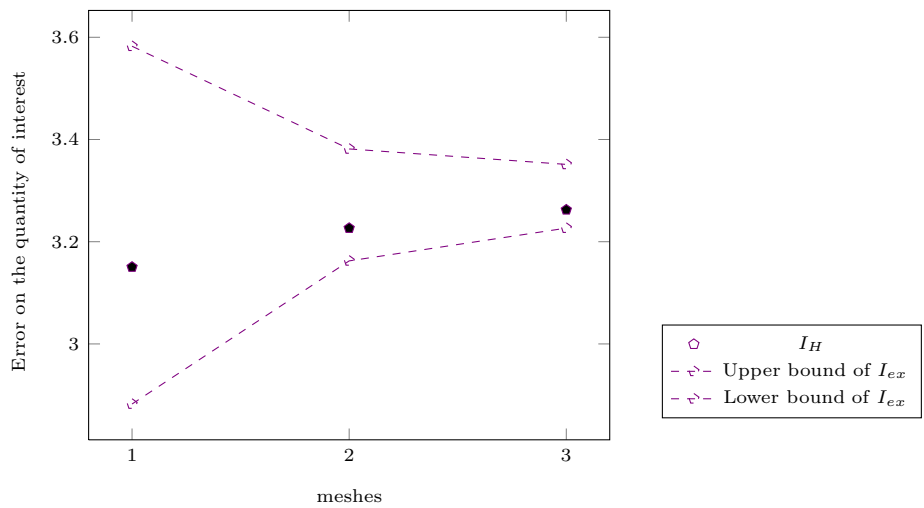

\vspace{-6pt}
\section{Conclusion}\label{sec:ccl}
\vspace{-2pt}
In this paper, we proposed a strict lower bound of the error in a substructured 
context which can be computed in parallel using the admissible fields  
built for the computation of the upper bound. Moreover, a theorem gives a 
second lower bound that separates the algebraic error from the discretization 
error. As illustrated on mechanical examples, the lower bounds are quasi independent 
from the substructuring and are as accurate as the sequential lower bound. The 
examples also show the separation of sources of error for the lower bound.
Finally, we proposed an auto-adaptive strategy to steer the iterative solver by an objective of 
precision on a quantity of interest.  Benefiting from the separation of sources of error to avoid oversolving and the 
recycling of Krylov subspaces, the strategy automatically defines a sequence of optimized {solvings}.

\vspace{-6pt}

\subsection*{Acknoledgment}
The authour would like to thank Professor Pedro Diez (Universitat Polit\`{e}cnica de Catalunya) for the helpful and fruitful discussions.


\bibliography{Biblio}
\end{document}